\documentclass[11pt]{amsart}

\usepackage{amsmath, amsthm,afterpage}

\usepackage{amssymb,amscd,amstext,units,mathtools,color,xcolor,enumerate,bm,pifont,comment}
\usepackage[mathscr]{eucal}
\usepackage[T1]{fontenc}
\usepackage[utf8]{inputenc}
\usepackage[all]{xy}
\usepackage[hidelinks]{hyperref}

\usepackage{geometry}

\theoremstyle{plain}
\newtheorem*{cor*}{Corollary}
\newtheorem*{prop*}{Proposition}
\newtheorem*{thm*}{Theorem}
\newtheorem*{notation*}{Notation}
\newtheorem{thm}{Theorem}[section]
\newtheorem{cor}[thm]{Corollary}
\newtheorem{prop}[thm]{Proposition}
\newtheorem{lem}[thm]{Lemma}
\newtheorem{fact}[thm]{Fact}
\newtheorem{claim}[thm]{Claim}

\theoremstyle{definition}
\newtheorem{dfn}[thm]{Definition}
\newtheorem{rem}[thm]{Remark}
\newtheorem{notation}[thm]{Notation}
\newtheorem{ex}[thm]{Examples}

\numberwithin{equation}{subsection}

\newcommand{\forkindep}[1][]{%
  \mathrel{
    \mathop{
      \vcenter{
        \hbox{\oalign{\noalign{\kern-.3ex}\hfil$\vert$\hfil\cr
              \noalign{\kern-.7ex}
              $\smile$\cr\noalign{\kern-.3ex}}}
      }
    }\displaylimits_{#1}
  }
}  

\newtheorem{innercustomgeneric}{\customgenericname}
\providecommand{\customgenericname}{}
\newcommand{\newcustomtheorem}[2]{%
  \newenvironment{#1}[1]
  {%
   \renewcommand\customgenericname{#2}%
   \renewcommand\theinnercustomgeneric{##1}%
   \innercustomgeneric
  }
  {\endinnercustomgeneric}
}

\newcustomtheorem{customthm}{Theorem}

\title{Definably compact groups definable in real closed fields.{I}}

\author{Eliana Barriga}
\address{ Eliana Barriga\\ Universidad de los Andes, Colombia\\
University of Haifa, Israel}
\email{el.barriga44@uniandes.edu.co}



\begin{document}

\begin{abstract}

We study definably compact definably connected groups definable in a sufficiently saturated real closed field $R$. We introduce the notion of group-generic point for $\bigvee$-definable groups and show the existence of group-generic points for definably compact groups definable in a sufficiently saturated o-minimal expansion of a real closed field. We use this notion along with some properties of generic sets to prove that for every definably compact definably connected group $G$ definable in $R$ there are a connected $R$-algebraic group $H$, a definable injective map $\phi$ from a generic definable neighborhood of the identity of $G$ into the group $H\left(R\right)$ of $R$-points of $H$ such that $\phi$ acts as a group homomorphism inside its domain. This result is used in \cite{BADefComp-II} to prove that the o-minimal universal covering group of an abelian connected definably compact group definable in a sufficiently saturated real closed field $R$ is, up to locally definable isomorphisms, an open connected locally definable subgroup of the o-minimal universal covering group of the $R$-points of some $R$-algebraic group.
\end{abstract}

\maketitle

\section{Introduction}

This is the first of two papers around definably compact groups definable in real closed fields.

Definable groups in o-minimal structures have been intensively studied
in the last three decades, and it is a field of current research.
A \textit{real closed field} is an ordered field elementarily equivalent to
the real ordered field $\mathbb{R}$; for instance, $\mathbb{R}$,
the real algebraic numbers $\mathbb{R}_{\textrm{alg}}$, the $\aleph_{1}$-saturated
hyperreal numbers $^{*}R$, which has infinite and infinitesimal elements,
among other examples. By quantifier elimination in real closed fields
(Tarski-Seidenberg), the definable sets in a real closed field $R$
are the \textit{semialgebraic sets} over $R$; namely, sets that are finite
Boolean combination of sets of solutions of finitely many polynomial
equations and inequalities over $R$. Since a real closed field is
an o-minimal structure (i.e., an ordered structure for which every
definable subset of its universe is a finite union of points and intervals,
see e.g., \cite{LVD}), then semialgebraic groups over a real closed
field can be seen as a generalization of the semialgebraic groups
over the real field $\mathbb{R}$, and also as a particular case of
the groups definable in an o-minimal structure.

There is a closed relation between groups definable in
a field $F$ and $F$-algebraic groups. Given an $F$-algebraic group
$H$, the group of $F$-points $H\left(F\right)$ is a definable group
in $F$. When $F$ is an algebraically closed field, every definable
group in $F$ is $F$-definably isomorphic, as an $F$-definable group, to some $F$-algebraic
group (\cite{BousII,LouChunkThm}); this fact is a version of Weil's
theorem that asserts that any $F$-algebraic group can be recovered
from birational data \cite{We}. However, when $F$ is real closed,
there are semialgebraic groups over $F$ that are not $F$-definably
isomorphic to $H\left(F\right)$ for any $F$-algebraic group $H$
(e.g., consider the group $\left(\left[0,1\right)\subseteq F,+_{mod\,1}\right)$).

Hrushovski and Pillay formulated in \cite{HP1,HP2} a relationship between a semialgebraic group $G$ over a real closed field $R$ and the set of $R$-points $H\left(R\right)$ of some $R$-algebraic group $H$. It roughly states that although the group operation of a semialgebraic group is given by a semialgebraic function, it is locally given by a rational function. More specifically, it assures the following.

\begin{fact}\label{F:ThmA}\cite[Thm. A]{HP1,HP2}
Let $G$ be a definably connected semialgebraic group over a real closed field $\mathcal{R}=\left(R,<,+,0,\cdot,1\right)$. Then there are a connected $R$-algebraic group $H$, a semialgebraic neighborhood $U$ of the identity of $G$, and a semialgebraic homeomorphism $f:U\rightarrow f\left(U\right)\subseteq H\left(R\right)$, where $H\left(R\right)$ is the set of $R$-points of $H$, such that $x,y,xy\in U$ implies $f\left(xy\right)=f\left(x\right)f\left(y\right)$.
\end{fact}

Where by a \textit{definably connected group} definable in $\mathcal{R}$ we mean that $G$ has no proper ($\mathcal{R}$)-definable subgroups of finite index.

Nevertheless, the neighborhood around the identity of $G$ given by the above Hrushovski and Pillay's result could not give enough information about $G$; for instance, if $U$ is too small. Consider the following example, the group $\left(\left[0,1\right)\subseteq R,+_{mod\, 1}\right)$ with addition modulo $1$ is locally homomorphic to $\left(R,+\right)$, where $\left(R,<,+,0,\cdot,1\right)$ is a $\aleph_{1}$-saturated real closed field. More precisely, the definable bijection $f:\left[0,\beta\right)\cup\left(1-\beta,1\right)\rightarrow\left(-\beta,\beta\right)$ defined by $f\left(x\right)=x$ if $x\in\left[0,\beta\right)$, or $f\left(x\right)=x-1$ if $x\in\left(1-\beta,1\right)$, where $0<\beta\ll\frac{1}{n}$ for every $n\in\mathbb{N}$ (i.e., $\beta$ is a positive infinitesimal), is a local homomorphism between $\left(\left[0,1\right)\subseteq R,+_{mod\, 1}\right)$ and $\left(R,+\right)$, where by a \textit{local homomorphism} between two groups we mean a map between some neighborhoods of their identities that acts as a group homomorphism inside its domain (see Def. \ref{D:homomoloc}). But $U=\left[0,\beta\right)\cup\left(1-\beta,1\right)$ cannot cover $G$ with finitely many group translates, and even the subgroup $\left\langle U\right\rangle $ generated by $U$ cannot say nothing about the torsion of $G$.

Fortunately, the definably compactness (see Def. \ref{D:TopNotLdGps}) of $G$ allows us to obtain a local homomorphism between $G$ and $H(R)$ whose domain is a generic definable set in $G$.

From now on, we will follow the next conventions. By a \textit{sufficiently saturated structure} we mean a $\kappa$-saturated structure for some sufficiently large cardinal $\kappa$. By a \textit{type-definable set}
in a sufficiently saturated structure $\mathcal{M}$ we mean a subset of $M^{n}$ that is the intersection
of less than $\kappa$-many definable sets. And given a group $G$ $\bigvee$-definable in an o-minimal structure, by $G^{00}$ we denote the smallest type-definable subgroup of $G$ of index $<\kappa$; if $G$ is definable, then  $G^{00}$ exists, by \cite{Shelah}.

In this paper we prove the next theorem.

\begin{customthm}{\ref{T:TA*}}
\textit{Let $G$ be a definably compact definably connected group definable in a sufficiently saturated real closed field $R$. Then there are \begin{enumerate}[(i)]
\item a connected $R$-algebraic group $H$ such that $\dim\left(G\right)=\dim\left(H\left(R\right)\right)=\dim\left(H\right)$,   \item a definable $X\subseteq G$ such that $G^{00}\subseteq X$,
\item a definable homeomorphism $\phi:X\subseteq G\rightarrow \phi\left(X\right)\subseteq H\left(R\right)$ such that $\phi$ and $\phi^{-1}$ are local homomorphisms.
\end{enumerate}}
\end{customthm}

To prove the above result we introduce the notion of group-generic point in  Section \ref{S:3}. An element $a$ of a group $G$ definable over $A\subseteq M$ is called \textit{group-generic} of $G$ over $A$ if every $A$-definable $X\subseteq G$ with $a\in X$ is \textit{generic} in $G$ (namely, $X$ covers $G$ by finitely many group translates), where $\mathcal{M}=\left(M,<,\ldots\right)$ is a sufficiently saturated o-minimal structure. We show the existence of group-generic points in definably compact groups definable in a sufficiently saturated o-minimal expansion of a real closed field (Prop. \ref{P:SupL1}) as well as we establish some properties of generic, group-generic points, and generic sets. With these tools we adapt the proof of \cite[Prop. 3.1]{HP1} to obtain a strong version of the group configuration result for definably compact groups (Prop. \ref{P:3.1*}), which is one of the main ingredients for Theorem \ref{T:TA*}.

In the second paper (\cite{BADefComp-II}) we combine Theorem \ref{T:TA*} and a study of locally definable covering homomorphisms for locally definable groups to prove the following: if $G$ is an abelian definably compact definably connected group $G$ definable in a sufficiently saturated real closed field, then its o-minimal universal covering group $\widetilde{G}$ is definably isomorphic, as a locally definable group, to a connected open locally definable subgroup of the o-minimal universal covering group $\widetilde{H\left(R\right)^{0}}$ of the
group $H\left(R\right)^{0}$ for some connected $R$-algebraic group $H$.

This research is part of my PhD thesis at the Universidad de los Andes, Colombia and University of Haifa, Israel.

\subsection{The structure of the paper}

Section \ref{S:2} contains some basic background used throughout the paper. Group-generic points are introduced and studied in Section \ref{S:3}. We define group-generic points for $\bigvee$-definable groups and show their existence in definably compact groups definable in a sufficiently saturated o-minimal expansion of a
real closed field as well as we establish some of their properties and connections with generic points and generic sets. In Section \ref{S:4}, we show a group configuration proposition for
definably compact groups (Prop. \ref{P:3.1*}) used in the proof of the main result of this paper: Theorem \ref{T:TA*}, which is proved in Section \ref{S:5}.

\begin{notation*}\label{N:}
Our notation and any undefined term that we use from model theory, topology, or algebraic geometry are generally standard. For a group $G$ whose group operation is written multiplicatively, we use the following notation $\prod_{n}X=\underset{n\textrm{-times}}{\underbrace{X\cdot\ldots\cdot X}}$,
and $X^{n}=\left\{ x^{n}:x\in X\right\} $ for any $n\in\mathbb{N}$.
\end{notation*}

\section{Preliminaries}\label{S:2}

Familiarity with basic facts about o-minimality is assumed (they
can be found in \cite{KniPiSte86}, \cite{PiSte86}, and \cite{LVD}).

Given a first-order structure $\mathcal{M}$ with universe $M$, we
say that a set $C\subseteq M^{k}$ is \textit{definable} in $\mathcal{M}$
over $A\subseteq M$ if there is a first order formula $\psi\left(x\right)$
with parameters from $A$ such that $C=\left\{ c\in M^{k}:\mathcal{M}\models\psi\left(c\right)\right\} $.
A set $C\subseteq M^{k}$ is definable in $\mathcal{M}$ if it is definable
over $M$. A \textit{function} $f:C\subseteq M^{k}\rightarrow M^{n}$
\textit{is definable} if its graph is a definable set. A group $\left(G,\cdot\right)$
is definable if $G$ is a definable set and its group multiplication
is a definable function.

In an o-minimal structure $\mathcal{M}=\left(M,<,\ldots\right)$ with
$C\subseteq M^{k}$ definable in $\mathcal{M}$, we define the  \textit{(geometric) dimension of} $C$, $\dim\left(C\right)$, as the maximal $n\leq m$ such that the projection of $C$ onto $n$ coordinates contains an open set of $M^{n}$, where $M^{n}$ has the product topology induced by the order topology on $M$.

\textit{From now until the end of this section, let $\mathcal{M}=\left(M,<,\ldots\right)$ be a sufficiently
saturated o-minimal structure.}

\subsection{Algebraic dimension and generic points}\label{S:3.1}

Recall that for $A\subseteq M$ and $b\in M$, $b$ is in the \textit{algebraic closure of}
$A$ ($b\in acl\left(A\right)$) if $b$ is an element of a finite
$A$-definable set. And $b$ is in the \textit{definable closure of}
$A$ ($b\in dcl\left(A\right)$) if the singleton $\left\{ b\right\} $
is $A$-definable. We can consider in this definitions of algebraic and definable closure finite tuples from $M$ instead of elements of $M$ with exactly the same definitions.

By the Exchange Lemma (\cite[Thm. 4.1]{PiSte86}), we can define a model theoretic notion of
dimension.

\begin{dfn}\label{D:}

Let $A\subseteq M$ and a tuple $a\in M^{n}$. The ($acl$-)\textit{dimension of}
$a$ over $A$, $\dim\left(a/A\right)$, is the cardinality of any
maximal $A$-algebraically independent subtuple of $a$. If $p\in S\left(A\right)$,
then $\dim\left(p\right)=\dim\left(a/A\right)$ for any tuple $a\in M^{n}$
realising $p$.

\end{dfn}

We recall some properties of this notion of dimension.

\begin{fact}\label{F:AlgDimProper}\cite[Lemma 1.2]{Pi}
Let $A,B\subseteq M$ and tuples $a\in M^{n}$ and $b\in M^{m}$.
\begin{enumerate}[(i)]

\item If $A\subseteq B$, then $\dim\left(a/A\right)\geq\dim\left(a/B\right)$.

\item (Additivity) $\dim\left(ab/A\right)=\dim\left(a/Ab\right)+\dim\left(b/A\right)$.

\item (Symmetry) $\dim\left(a/Ab\right)=\dim\left(a/A\right)$ $\Leftrightarrow$
$\dim\left(b/Aa\right)=\dim\left(b/A\right)$.

\item Let $p\in S\left(A\right)$. If $A\subseteq B$, there is $q\in S\left(B\right)$
such that $q\supseteq p$ and $\dim\left(p\right)=\dim\left(q\right)$.

\end{enumerate}

\end{fact}

\begin{dfn}\label{D:Indep}
Let $A\subseteq M$ and tuples $a\in M^{n}$ and $b\in M^{m}$.
\begin{enumerate}[(i)]

\item $a$ is \textit{independent} from $b$ over $A$, denoted
by $a\forkindep[A]b$, if $\dim\left(a/A\right)=\dim\left(a/Ab\right)$,
which is also expressed by saying that $\textrm{tp}\left(a/Ab\right)$
does not fork over $A$.

\item Let $X\subseteq M^{n}$ $A$-definable and $a\in X$. $a$
is a \textit{generic} point of $X$ over $A$ if $\dim\left(a/A\right)=\dim\left(X\right)$.

\end{enumerate}

\end{dfn}

Note, by \cite[Lemma 1.4]{Pi}, that the (geometric) dimension
of an $A$-definable set $X\subseteq M^{n}$ satisfies $\dim\left(X\right)=\max\left\{ \dim\left(a/A\right):a\in X\right\} $.

Since $\mathcal{M}$ is sufficiently saturated, we have that for every $X\subseteq M^{n}$ definable over $A\subseteq M$, with $\left|A\right|<\kappa$, there is $a\in X$ generic of $X$ over $A$.

\subsection{$\bigvee$-definable groups}\label{S:2.2}

\begin{dfn}\label{D:VdefGps}

A $\bigvee$-\textit{definable group}
is a group $\left(\mathcal{U},\cdot\right)$ whose universe is an
union $\mathcal{U}=\bigcup_{i\in I}Z_{i}$ of $\mathcal{M}$-definable
subsets of $M^{n}$ for some fixed $n$, all defined over $A\subseteq M$ with $\left|A\right|<\kappa$ such that for every $i,j\in I$
\begin{enumerate}[(i)]

\item there is $k\in I$ such that $Z_{i}\cup Z_{j}\subseteq Z_{k}$
(i.e., the union is directed), and

\item the group operation $\cdot\mid_{Z_{i}\times Z_{j}}$ and group
inverse $\left(\cdot\right)^{-1}\mid_{Z_{i}}$ are $\mathcal{M}$-definable
maps into $M^{n}$.
\end{enumerate}
We say that $\left(\mathcal{U},\cdot\right)$ is \textit{locally $\mathcal{M}$-definable}
if $\left|I\right|$ is countable.

A map between $\bigvee$-definable (locally definable) groups is called
$\bigvee$-\textit{definable} (\textit{locally $\mathcal{M}$-definable})
if its restriction to any $\mathcal{M}$-definable set is a $\mathcal{M}$-definable
map.

We define $\dim\left(\mathcal{U}\right)=\max\left\{ \dim\left(Z_{i}\right):i\in I\right\}$.

An element $a\in\mathcal{U}$ is \textit{generic} of $\mathcal{U}$ over $A$
is $\dim\left(a/A\right)=\dim\left(\mathcal{U}\right)$.
\end{dfn}

\begin{ex}\label{E:Exldgps}
\begin{enumerate}[(i)]
\item Let $G$ be a $\mathcal{M}$-definable group, and $X\subseteq G$
definable containing the identity element of $G$. Then the subgroup
$\mathcal{U}=\left\langle X\right\rangle =\bigcup_{n\in\mathbb{N}^{\times}}\prod_{n}X\cdot X^{-1}$
of $G$ generated by $X$ is a locally $\mathcal{M}$-definable group.
Then, in particular, every countable group is a locally definable
group in any structure as well as the commutator subgroup $\left[G,G\right]$ of a $\mathcal{M}$-definable
group $G$.

\item The o-minimal universal covering group (see \cite{EdPan}) of a connected locally $\mathcal{M}$-definable group exists and is a locally $\mathcal{M}$-definable group.

\item (\cite{PPant12}) Let $\left(G,<,+\right)$ be a sufficiently
saturated ordered divisible abelian group, and in it take an infinite
increasing sequence of elements $0<a_{1}<a_{2}<\cdots$ such that
$na_{i}<a_{i+1}$ for every $n\in\mathbb{N}$. The subgroup $\mathcal{U}=\bigcup_{i}\left(-a_{i},a_{i}\right)$
of $G$ is a $\bigvee$-definable group. This group has the distinction
of having no $\mathcal{U}^{00}$, and is not definably generated.

\end{enumerate}
\end{ex}

Any $\bigvee$-definable group in $\mathcal{M}$ can be endowed with a topology $\tau$ making it into a topological group (\cite[Prop. 2.2]{PS00}). This fact was a generalization
by Peterzil and Starchenko \cite{PS00} of the known result for the
definable groups by Pillay \cite{Pi}. In case $\mathcal{M}$
expands the reals, that topology makes any definable group into a
real Lie group. Moreover, any $\bigvee$-definable homomorphism between
two $\bigvee$-definable groups is continuous with respect to their
$\tau$ topologies \cite[Lemma 2.8]{PS00}, and any $\bigvee$-definable subgroup $\mathcal{W}$
of a $\bigvee$-definable group $\mathcal{U}$ is $\tau$-closed in $\mathcal{W}$, and if $\mathcal{W}$ and $\mathcal{U}$ have the same dimension, then $\mathcal{W}$ is also $\tau$-open in $\mathcal{W}$ \cite[Lemma 2.6]{PS00}. \textit{For the rest of the paper any topological property of $\bigvee$-definable groups refers to this $\tau$ topology.}

\begin{dfn}\label{D:TopNotLdGps}
Let $\mathcal{U}$ be a $\bigvee$-definable group.
\begin{enumerate}[(i)]
\item $\mathcal{U}$ is ($\tau$-)\textit{connected} if $\mathcal{U}$
has no nonempty proper ($\tau$-)clopen subset such that its intersection
with any definable subset of $\mathcal{U}$ is definable.

\item $\mathcal{U}$ is ($\tau$-)\textit{definably compact} if every definable
path $\gamma:\left(0,1\right)\rightarrow\mathcal{U}$ has limits points
in $\mathcal{U}$ (where the limits are taken with respect to the
$\tau$-topology).
\end{enumerate}
\end{dfn}

Note that if $\mathcal{U}$ is a definable group, the above definition of connectedness agrees with the known notion of definable connectedness for definable groups.

\subsection{Generic sets in $\bigvee$-definable groups}\label{S:3.2}

\begin{dfn}\label{D:}
Let $\mathcal{U}$ be a $\bigvee$-definable group. A set $X\subseteq\mathcal{U}$
is \textit{left (right) generic} in $\mathcal{U}$ if less than $\kappa$-many left (right) group translates of $X$ cover $\mathcal{U}$. $X$
is \textit{generic} if it is both left and right generic, and $X$
is called $n$-\textit{generic} if $n$-group translates of $X$
cover $\mathcal{U}$.
\end{dfn}

Thus, by saturation, a definable subset generic in a definable group covers the group in finitely many group translates. Moreover, by \cite[Fact 2.3(2)]{PPant12}, any left generic definable subset of a connected $\bigvee$-definable group $\mathcal{U}$ generates $\mathcal{U}$.

Some examples of generic definable subsets of a definable group $G$
are the large subsets in $G$; namely, a definable set $Y\subseteq G$
such that $\dim\left(G\setminus Y\right)<\dim\left(G\right)$, this
fact was proved by Pillay in \cite[Lemma 2.4]{Pi}. Also, note that
for the additive group $\left(M,+\right)$ a definable generic set
$X\subseteq M$ is generic if and only if $M\setminus X$ is bounded.
in $M$ (\cite[Remark 3.3]{PePi07}).

In case $G$ is a definably
compact definably connected group, by \cite[Prop. 4.2]{HPePiNIP},
for any definable set $X\subseteq G$, $X$ is left generic if and
only if $X$ is right generic, so we just say generic.

\begin{fact}\cite[Thm. 3.7]{PePi07}\label{F:T3.7}
Assume $G$ is a definably connected group definable in a sufficiently saturated o-minimal expansion of a real closed field, and $X\subseteq G$ a definable set whose closure in $G$ is definably compact. If $X$ is not left generic in $G$ then $G\setminus X$ is right generic in $G$. \end{fact}

\subsection{Local homomorphisms}\label{S:4.2}

\begin{dfn}\label{D:homomoloc}
Let $G_{1}$ and $G_{2}$ be two topological groups, $X\subseteq G_{1}$ a neighborhood of the identity of $G_{1}$, and $\theta:X\rightarrow G_{2}$ a map. $\theta$ is called a \textit{local homomorphism} if $x,y,xy\in X$ implies $\theta\left(xy\right)=\theta\left(x\right)\theta\left(y\right)$. We say that an injective map $\theta:X\subseteq G_{1}\rightarrow G_{2}$ is a \textit{local homomorphism in both directions} if $\theta:X\rightarrow G_{2}$ and $\theta^{-1}:\theta\left(X\right)\rightarrow X$ are local homomorphisms.
\end{dfn}
\begin{rem}\label{R:inverseneednota loc homo}
Note that if $\theta:X\subseteq G_{1}\rightarrow G_{2}$ is an injective local homomorphism between the groups $G_{1}$, $G_{2}$, then $\theta^{-1}:\theta\left(X\right)\rightarrow X$ need not be a local homomorphism; for instance, consider the groups $G_{1}=\left(\mathbb{R},+\right)$, $G_{2}=\left(\left[0,1\right),+_{mod\, 1}\right)$, and $\theta:\left[-\frac{1}{8},\frac{1}{8}\right]\subseteq\mathbb{R}\rightarrow\left[0,1\right)$ the map $\theta\left(x\right)=\pi\left(3x\right)$ where $\pi:\mathbb{R}\rightarrow\left[0,1\right):t\mapsto t\,\textrm{mod}1$. Then $\theta$ is an injective local homomorphism, but $\theta^{-1}:\left[0,\frac{3}{8}\right]\cup\left[\frac{3}{8},1\right)=\theta\left(\left[-\frac{1}{8},\frac{1}{8}\right]\right)\rightarrow\left[-\frac{1}{8},\frac{1}{8}\right]\subseteq\mathbb{R}$ is not a local homomorphism. For this note that, for example $\frac{5}{8}+_{mod\, 1}\frac{5}{8}=\frac{1}{4}\in\theta\left(\left[-\frac{1}{8},\frac{1}{8}\right]\right)$, but $\theta^{-1}\left(\frac{5}{8}\right)+\theta^{-1}\left(\frac{5}{8}\right)=-\frac{1}{4}\notin\left[-\frac{1}{8},\frac{1}{8}\right]$. In Claim \ref{C:tetainvisloc_homomo} we formulate a necessary and sufficient condition on a local homomorphism $\theta$ in order for $\theta^{-1}$ to be a local homomorphism.
\end{rem}
\begin{claim}\label{C:tetainvisloc_homomo}
Let $\theta:X\subseteq G_{1}\rightarrow G_{2}$ be an injective local homomorphism between the groups $G_{1}$ and $G_{2}$. Then
\begin{enumerate}[(i)]
\item $\theta^{-1}:\theta\left(X\right)\subseteq G_{2}\rightarrow X\subseteq G_{1}$ is a local homomorphism if and only if for all $y_{1},y_{2}\in\theta\left(X\right)$ if $y_{1}\cdot y_{2}\in\theta\left(X\right)$, then $\theta^{-1}\left(y_{1}\right)\cdot\theta^{-1}\left(y_{2}\right)\in X$.
\item If there is $X^{\prime}\subseteq X$ such that $X^{\prime}\cdot X^{\prime}\subseteq X$, then $\theta\mid_{X^{\prime}}^{-1}:\theta\left(X^{\prime}\right)\rightarrow X^{\prime}$ is a local homomorphism.
\end{enumerate}
\end{claim}
\begin{proof}
(i) Let $y_{1},y_{2},y_{1}\cdot y_{2}\in\theta\left(X\right)$ such that $\theta^{-1}\left(y_{1}\right)\cdot\theta^{-1}\left(y_{2}\right)\in X$, then $\theta\left(\theta^{-1}\left(y_{1}\right)\cdot\theta^{-1}\left(y_{2}\right)\right)=y_{1}\cdot y_{2}$ $\Leftrightarrow$$\theta^{-1}\left(y_{1}\cdot y_{2}\right)=\theta^{-1}\left(y_{1}\right)\cdot\theta^{-1}\left(y_{2}\right)$. The other direction is clear.

(ii) Let $\theta\left(x_{1}\right)$, $\theta\left(x_{2}\right)$, $\theta\left(x_{1}\right)\cdot\theta\left(x_{2}\right)\in\theta\left(X^{\prime}\right)$. Since $X^{\prime}\cdot X^{\prime}\subseteq X$, $\theta\left(x_{1}\cdot x_{2}\right)=\theta\left(x_{1}\right)\cdot\theta\left(x_{2}\right)$, then $\theta^{-1}\left(\theta\left(x_{1}\right)\right)\cdot\theta^{-1}\left(\theta\left(x_{2}\right)\right)=x_{1}\cdot x_{2}=\theta^{-1}\left(\theta\left(x_{1}\right)\cdot\theta\left(x_{2}\right)\right)\in\theta^{-1}\left(\theta\left(X^{\prime}\right)\right)=X^{\prime}$.
\end{proof}

\begin{rem}\label{R:R2}
Let $G_{1}$ and $G_{2}$ be two groups, and $W\subseteq G_{1}$. Let $\theta:W\rightarrow G_{2}$ be an injective local homomorphism. If there is $W^{\prime}\subseteq W$ such that $W^{\prime}W^{\prime}\subseteq W$, then $\left(\theta\mid_{W^{\prime}}\right)^{-1}:\theta\left(W^{\prime}\right)\rightarrow W^{\prime}\subseteq G_{1}$ is a local homomorphism.
\end{rem}
\begin{proof}
Let $x,y\in W^{\prime}$. Since $W^{\prime}W^{\prime}\subseteq W$, $\theta\left(xy\right)=\theta\left(x\right)\theta\left(y\right)$, then $\theta^{-1}\left(\theta\left(x\right)\right)\theta^{-1}\left(\theta\left(y\right)\right)=xy=\theta^{-1}\left(\theta\left(x\right)\theta\left(y\right)\right)$.
\end{proof}

\section{Group-generic points for $\bigvee$-definable groups}\label{S:3}

\textit{From now until the end of this chapter, $\mathcal{M}$ is a sufficiently saturated o-minimal expansion of a real closed field.}

\begin{dfn}\label{D:Gp-generic}
Let $\mathcal{U}$ be a $\bigvee$-definable group over $A\subseteq M$ and $a\in \mathcal{U}$.
\begin{enumerate}[(i)]

\item $a$ is a \textit{left (right) group-generic point} of $\mathcal{U}$
over $A$ if every $A$-definable $X\subseteq \mathcal{U}$ with $a\in X$
is left (right) in $\mathcal{U}$. $a$ is \textit{group-generic} if it is
both left and right generic.

\item A type $p$ is \textit{generic in} $\mathcal{U}$ if for every formula
$\varphi\in p$, $\varphi$ defines a generic subset in $\mathcal{U}$.

\end{enumerate}

Therefore, $a\in \mathcal{U}$ is group-generic of $\mathcal{U}$ over $A$ if and only
if $\textrm{tp}\left(a/A\right)$ is generic in $\mathcal{U}$.

\end{dfn}

\begin{rem}\label{R:}
Let $G$ be a group definable over $A\subseteq M$, and $a\in G$.
If $a$ is group-generic of $G$ over $A$, then $a$ is generic of
$G$ over $A$.
\end{rem}
\begin{proof}
Suppose that there is an $A$-definable set $Y$ with $a\in Y$ and
$\dim\left(Y\right)<\dim\left(G\right)$, then $\dim\left(Y\cap G\right)<\dim\left(G\right)$,
then $Y\cap G$ cannot be generic in $G$, but this contradicts the
group-genericity of $a$.
\end{proof}

\subsection{Basics on group-generic points}

Below we will discuss some properties of group-generic points and their relationships with generic points and generic sets.

The next fact is a consequence of \cite[Thm. 3.7]{PePi07}.

\begin{fact}\cite{PePi07}\label{F:ExistOfGenTypes}
Let $G$ be a definably compact definably connected definable group. Then:
\begin{enumerate}[(i)]
  \item The union of two nongeneric definable subsets in $G$ is also nongeneric in $G$.
  \item The set \[ \mathcal{I}=\left\{ X\subseteq G:X\,\textrm{is definable and nongeneric in }G\right\}\] is an ideal of $\left(Def\left(G\right),\cup,\cap\right)$, the Boolean algebra of definable subsets of $G$.
  \item There is a complete generic type $p\left(x\right)\in S^{\mathcal{M}}\left(A\right)$ in $G$.
\end{enumerate}
\end{fact}

\begin{prop}\label{P:SupL1}
Let $G$ be a definably compact definably connected definable over $A\subseteq M$ group. Then:
\begin{enumerate}[(i)]
  \item If $\left|A\right|<\kappa$, then there is a group-generic point $a$ of $G$ over $A$.
  \item Let $p\in S^{\mathcal{M}}\left(A\right)$ be a generic type in $G$. Then for any $B$ such that $A\subseteq B\subseteq M$ there is a generic type $q\in S^{\mathcal{M}}\left(B\right)$ in $G$ such that $q\supseteq p$ and $q\vert_{A}=p$.
\end{enumerate}
\end{prop}
\begin{proof}
(i) By Fact \ref{F:ExistOfGenTypes}(iii), there is a complete generic type $p\left(x\right)\in S^{\mathcal{M}}\left(A\right)$ in $G$. Then, by saturation, there is $a\models p$, so $\textrm{tp}\left(a/A\right)=p$ is generic in $G$; i.e., $a$ is group-generic of $G$ over $A$.

(ii) Suppose that $G=\varphi\left(\mathcal{M},a\right)$ for some $\mathcal{L}_{A}$-formula
$\varphi$ with $a\subseteq A$. Let \[ \mathcal{I}_{B}=\left\{ X\subseteq G:\textrm{ \ensuremath{X} is \ensuremath{\mathcal{M}}-definable over \ensuremath{B} and nongeneric in }G\right\} \]
and \[ \Phi_{B}=\left\{ \lnot\psi\land\varphi:\psi\textrm{ is a \ensuremath{\mathcal{L}_{B}}-formula that defines in \ensuremath{\mathcal{M}} a set in \ensuremath{\mathcal{I}_{B}}}\right\} .\]
Let $\overline{p}\coloneqq p\cup\Phi_{B}$. Then $\overline{p}$ is
a partial type since if there are $\left\{ \theta_{i}\right\} _{i<k_{1}}\subseteq p$,
$\left\{ \lnot\psi_{j}\land\varphi\right\} _{j<k_{2}}\subseteq\Phi_{B}$
such that $\bigwedge_{i<k_{1}}\theta_{i}\land\bigwedge_{j<k_{2}}\lnot\psi_{j}\land\varphi$
is not satisfiable, then $\bigwedge_{i<k_{1}}\theta_{i}\rightarrow\bigvee_{j<k_{2}}\psi_{j}$
is satisfiable, but $\bigwedge_{i<k_{1}}\theta_{i}\in p$ and by Fact \ref{F:ExistOfGenTypes}(i), $\bigvee_{j<k_{2}}\psi_{j}$ is nongeneric, so there is a formula in $p$ that implies a nongeneric formula, which
contradicts the genericity of $p$. Hence, $\overline{p}$ is finitely
satisfiable. Now, let $q\left(x\right)$ be any complete type in $S^{\mathcal{M}}\left(B\right)$
such that $q\supseteq\overline{p}$, then $q$ is a generic type in $G$ and $q\vert_{A}=p$. This finishes the proof of (ii).
\end{proof}

\begin{cor}\label{C:SupL2}
Let $G$ be a definably compact definably connected group definable over $A\subseteq M$ with $\left|A\right|<\kappa$. Let $a\in G$ be a group-generic element of $G$ over $A$ , and $c$ a finite tuple from $M$. Then,
\begin{enumerate}[(i)]

  \item there is $a^{\prime}\in G$ such that $\textrm{tp}^{\mathcal{M}}\left(a^{\prime}/A\right)=\textrm{tp}^{\mathcal{M}}\left(a/A\right)$ and $a^{\prime}$ is a group-generic element of $G$ over $Ac$.
  \item There is $c^{\prime}$ a finite tuple from $M$ such that $\textrm{tp}^{\mathcal{M}}\left(c^{\prime}/A\right)=\textrm{tp}^{\mathcal{M}}\left(c/A\right)$ and $a$ is a group-generic element of $G$ over $Ac^{\prime}$.
  \item There is $c^{\prime}$ a generic element of $G$ over $Aa$ such that $a$ is group-generic of $G$ over $Ac^{\prime}$.
  \item Let $b\in G$ group-generic of $G$ over $Ac$. If $c^{\prime}$ is a finite tuple from $M$ such that $\textrm{tp}^{\mathcal{M}}\left(c^{\prime}/Ab\right)=\textrm{tp}^{\mathcal{M}}\left(c/Ab\right)$, then $b$ is group-generic of $G$ over $Ac^{\prime}$.
  \item Let $b\in G$ group-generic of $G$ over $A$ and $a$ group-generic of $G$ over $Ab$. Then there is $c^{\prime}$ a finite tuple from $M$ such that $\textrm{tp}^{\mathcal{M}}\left(c^{\prime}/A\right)=\textrm{tp}^{\mathcal{M}}\left(c/A\right)$, and $b$ and $a$ are group-generic of $G$ over $Ac^{\prime}$ and $Abc^{\prime}$, respectively.
   \item Let $b\in G$ group-generic of $G$ over $A$ and $a$ group-generic of $G$ over $Ab$. Then there is $c^{\prime}$ generic of $G$ over $Aab$ such that $b$ and $a$ are group-generic of $G$ over $Ac^{\prime}$ and $Abc^{\prime}$, respectively.
\end{enumerate}
\end{cor}
\begin{proof}
(i) Since $a$ is group-generic of $G$ over $A$, $\textrm{tp}^{\mathcal{M}}\left(a/A\right)$
is a complete generic type in $G$. Let $B=Ac$. Then by Prop. \ref{P:SupL1}(ii), there
is a generic type $q\in S{}^{\mathcal{M}}\left(B\right)$ in $G$
such that $q\supseteq \textrm{tp}^{\mathcal{M}}\left(a/A\right)$ and $q\vert_{A}=\textrm{tp}^{\mathcal{M}}\left(a/A\right)$.
As $\mathcal{M}$ is sufficiently saturated, there is $a^{\prime}\subseteq M$
such that $a^{\prime}\models q$, thus $\textrm{tp}^{\mathcal{M}}\left(a^{\prime}/B\right)=q$.
Since $q\vert_{A}=\textrm{tp}^{\mathcal{M}}\left(a/A\right)$, $\textrm{tp}^{\mathcal{M}}\left(a^{\prime}/A\right)=\textrm{tp}^{\mathcal{M}}\left(a/A\right)$.

(ii) By (i), there is $a^{\prime}\subseteq M$ such that $\textrm{tp}^{\mathcal{M}}\left(a^{\prime}/A\right)=\textrm{tp}^{\mathcal{M}}\left(a/A\right)$
and $a^{\prime}$ is a group-generic element of $G$ over $Ac$. As $\mathcal{M}$
is sufficiently saturated, $\textrm{tp}^{\mathcal{M}}\left(a^{\prime}/A\right)=\textrm{tp}^{\mathcal{M}}\left(a/A\right)$
if and only if there is $f\in Aut\left(\mathcal{M}/A\right)$ such
that $f\left(a^{\prime}\right)=a$. Then $a^{\prime}$ is a group-generic
element of $G$ over $Ac$ if and only if $a$ is a group-generic element
of $G$ over $Af\left(c\right)$, so with $c^{\prime}=f\left(c\right)$
we obtain the desired conclusion.

(iii) Let $c$ be a generic element of $G$ over $A$. By (ii), there
is $c^{\prime}\subseteq M$ such that $\textrm{tp}^{\mathcal{M}}\left(c^{\prime}/A\right)=\textrm{tp}^{\mathcal{M}}\left(c/A\right)$
and $a$ is a group-generic element of $G$ over $Ac^{\prime}$. As $a$
is group-generic of $G$ over $Ac^{\prime}$, $a\forkindep[A]c^{\prime}$. Since
$\textrm{tp}^{\mathcal{M}}\left(c^{\prime}/A\right)=\textrm{tp}^{\mathcal{M}}\left(c/A\right)$
and $c$ is generic of $G$ over $A$, then $c^{\prime}$ is generic
of $G$ over $A$, but $c^{\prime}\forkindep[A]a$, then $c^{\prime}$
is generic of $G$ over $Aa$.

(iv) First note the following.

\begin{claim}\label{C:f(b)isGG}
Let $b\in G$ group-generic of $G$ over $Ac$, and $f\in Aut\left(\mathcal{M}/A\right)$.
Then $f\left(b\right)$ is group-generic of $G$ over $Af\left(c\right)$.
\end{claim}
\begin{proof}
Let $\varphi\left(x,a^{\prime\prime},f\left(c\right)\right)$ be a
$\mathcal{L}_{Af\left(c\right)}$-formula with $a^{\prime\prime}$
a finite tuple from $A$ such that $\mathcal{M}\models\varphi\left(f\left(b\right),a^{\prime\prime},f\left(c\right)\right)$
and $\varphi\left(\mathcal{M},a^{\prime\prime},f\left(c\right)\right)\subseteq G$.
We will see that $\varphi\left(G,a^{\prime\prime},f\left(c\right)\right)$
is generic in $G$. As $f\left(b\right)\models\varphi\left(x,a^{\prime\prime},f\left(c\right)\right)$
if and only if $b\models\varphi\left(x,a^{\prime\prime},c\right)$,
then $\varphi\left(G,a^{\prime\prime},c\right)$ is generic in $G$.
From this it is easy to see that $\varphi\left(G,a^{\prime\prime},f\left(c\right)\right)$
is also generic in $G$. Thus $f\left(b\right)$ is group-generic
of $G$ over $Af\left(c\right)$.
\end{proof}

Now, as $\textrm{tp}^{\mathcal{M}}\left(c^{\prime}/Ab\right)=\textrm{tp}^{\mathcal{M}}\left(c/Ab\right)$,
then there is $f\in Aut\left(\mathcal{M}/Ab\right)$ such that $f\left(c\right)=c^{\prime}$.
Since $b\in G$ group-generic of $G$ over $Ac$, the above claim
yields $b=f\left(b\right)$ is group-generic of $G$ over $Ac^{\prime}$.

(v) By (ii), there is $c_{1}$ a tuple from $M$ such that $\textrm{tp}^{\mathcal{M}}\left(c_{1}/A\right)=\textrm{tp}^{\mathcal{M}}\left(c/A\right)$
and $b$ is group-generic of $G$ over $Ac_{1}$. Again by (ii), there
is $c^{\prime}$ a tuple from $M$ such that $\textrm{tp}^{\mathcal{M}}\left(c^{\prime}/Ab\right)=\textrm{tp}^{\mathcal{M}}\left(c_{1}/Ab\right)$
and $a$ is group-generic of $G$ over $Abc^{\prime}$. And by (iv),
$b$ is group-generic of $G$ over $Ac^{\prime}$.

(vi) Let $c$ be a generic element of $G$ over $A$. By (v), there is $c^{\prime}$
a tuple from $M$ such that $\textrm{tp}^{\mathcal{M}}\left(c^{\prime}/A\right)=\textrm{tp}^{\mathcal{M}}\left(c/A\right)$,
and $b$ and $a$ are group-generic of $G$ over $Ac^{\prime}$ and
$Abc^{\prime}$, respectively. Since $\textrm{tp}^{\mathcal{M}}\left(c^{\prime}/A\right)=\textrm{tp}^{\mathcal{M}}\left(c/A\right)$,
$a\forkindep[Ab]c^{\prime}$, and $b\forkindep[A]c^{\prime}$, then
$c^{\prime}$ is generic of $G$ over $Aab$.

\end{proof}

\begin{rem}\label{R:ggACL}

Let $G$ be a group definable over $A\subseteq M$. If $a$ is group-generic
of $G$ over $A$, then $a$ is group-generic of $G$ over $\textrm{acl}{}^{\mathcal{M}}\left(A\right)$.

\end{rem}

\begin{proof}

Let $c\in\textrm{acl}{}^{\mathcal{M}}\left(A\right)$, and assume
that $a\models\varphi\left(x,c\right)$ for some $\mathcal{L}_{c}$-formula
$\varphi\left(x,c\right)$ that defines a subset of $G$. We will
see that $\varphi\left(G,c\right)$ is generic in $G$. Since $c\in\textrm{acl}{}^{\mathcal{M}}\left(A\right)$
and $\mathcal{M}$ is an ordered structure, $\textrm{acl}{}^{\mathcal{M}}=\textrm{dcl}{}^{\mathcal{M}}$.
Then there is $c^{\prime}\in A$ such that $\gamma\left(\mathcal{M},c^{\prime}\right)=\left\{ c\right\} $.

Thus, if $\phi\left(x,c^{\prime}\right)=\left(\exists!z\right)\left(\varphi\left(x,z\right)\wedge\gamma\left(z,c^{\prime}\right)\right)$,
then $a\models\phi\left(x,c^{\prime}\right)$.

Since $a$ is group-generic of $G$ over $A$, $\phi\left(\mathcal{M},c^{\prime}\right)$
is generic in $G$. Therefore, there are $g_{1},\ldots,g_{n}\in G$
such that for every $g\in G$ there is $g^{\prime}\models\phi\left(\mathcal{M},c^{\prime}\right)$
such that $g=g_{i}\cdot g^{\prime}$ for some $i\in\left\{ 1,\ldots,n\right\} $.
But $\gamma\left(\mathcal{M},c^{\prime}\right)=\left\{ c\right\} ,$then
$g^{\prime}\models\varphi\left(x,c\right)$. Thus $\varphi\left(G,c\right)$
is generic in $G$.

\end{proof}

\begin{rem}\label{R:}
If $G$ is a group definable over $A\subseteq M$ and $a\in G$ is
generic of $G$ over $A$, then $a$ need not be group-generic
of $G$ over $A$. For instance, consider a sufficiently saturated
real closed field $\mathcal{R}=\left(R,<,+,0,\cdot,1\right)$ and
its additive group $G=\left(R,+\right)$. By saturation, there is
$\alpha\in\bigcap_{n\in\mathbb{N}}\left(0,\frac{1}{n}\right)$. Since
$\dim\left(\alpha/\emptyset\right)=tr.deg\left(\mathbb{Q}\left(\alpha\right):\mathbb{Q}\right)=1=\dim\left(G\right)$,
then $\alpha$ is generic of $G$ over $\emptyset$, but $\alpha$
is not group-generic of $G$ over $\emptyset$ since $\alpha\in\left(0,1\right)$
and $\left(0,1\right)$ cannot cover $G$ by finitely many group translates.
\end{rem}

\begin{rem}\label{R:}
Let $X\subseteq M^{n}$ definable over $A\subseteq M$. If $b\in X$
is generic of $X$ over $A$ and $a\in X$ is generic of $X$ over
$Ab$, then $b$ is generic of $X$ over $Aa$.
\end{rem}
\begin{proof}
Since $a$ is generic of $X$ over $Ab$, $a\forkindep[A]b$. By the
symmetry of the independence (Fact \ref{F:AlgDimProper}(iii)), $b\forkindep[A]a$,
so $b$ is generic of $X$ over $Aa$.
\end{proof}

\begin{rem}\label{R:RemGenPt}
If $G$ is a group definable over $A\subseteq M$, $b\in G$ is group-generic
of $G$ over $A$ and $a\in G$ is group-generic of $G$ over $Ab$,
then $b$ need not be group-generic of $G$ over $Aa$. For
instance, consider a sufficiently saturated real closed field $\mathcal{R}=\left(R,<,+,0,\cdot,1\right)$.
Let $\left[0,1\right)\subseteq R$ and $G=\left(\left[0,1\right),+_{mod\, 1}\right)$.

We will find $a,b\in G$ such that $b$ and $a$ are group-generic
of $G$ over $A$ and $Ab$, respectively, but with $b$ not group-generic
of $G$ over $Aa$.

Let $\varphi$ be a $\mathcal{L}_{A}$-formula that defines $G$.

Let
\[
\mathcal{I}_{A}=\left\{ \psi:\psi\textrm{ is a \ensuremath{\mathcal{L}_{A}}-formula and \ensuremath{\psi\left(\mathcal{R}\right)\subseteq G} is nongeneric in }G\right\} .
\]
Let $\theta_{n}\left(x\right)=0<x<\frac{1}{n}$ for $n\in\mathbb{N}\setminus\left\{ 0\right\} $,
and let
\[
\Gamma_{A}=\left\{ \lnot\psi\land\varphi,\theta_{n}:\psi\in\mathcal{I}_{A},n\in\mathbb{N}\setminus\left\{ 0\right\} \right\}.
\]

$\Gamma_{A}$ is a partial type because if $\left(\bigwedge_{i<k_{1}}\lnot\psi_{i}\land\varphi\wedge\bigwedge_{i<k_{2}}\theta_{j}\right)\left(\mathcal{R}\right)=\emptyset$,
then $\left(\bigwedge_{i<k_{2}}\theta_{j}\right)\left(\mathcal{R}\right)\subseteq\left(\bigvee_{i<k_{1}}\psi_{i}\right)\left(\mathcal{R}\right)$,
but the finite union of nongeneric definable subsets in $G$ is nongeneric,
so $\left(\bigvee_{i<k_{1}}\psi_{i}\right)\left(\mathcal{R}\right)$
cannot contain the generic set $\left(\bigwedge_{i<k_{2}}\theta_{j}\right)\left(\mathcal{R}\right)$,
thus $\Gamma_{A}$ is a partial type. Moreover, $\Gamma_{A}$ is generic
in $G$ because every formula in $\Gamma_{A}$ defines a generic subset
in $G$.

Now, let $p\in S^{\mathcal{R}}\left(A\right)$ with $p\supseteq\Gamma_{A}$,
then $p$ is also generic in $G$, otherwise if there is $\phi\in p$
defining a nongeneric subset in $G$, then $\lnot\phi\land\varphi,\phi\in p$,
but this is a contradiction.

By saturation, there is $b\models p$, thus $b$ is group-generic
of $G$ over $A$.

Following, we will show the existence of a group-generic point $a$
of $G$ over $Ab$ that is also a positive infinitesimal, but $b$
is not group-generic of $G$ over $Aa$.

Let
\[
\mathcal{I}_{Ab}=\left\{ \psi:\psi\textrm{ is a \ensuremath{\mathcal{L}_{Ab}}-formula and \ensuremath{\psi\left(\mathcal{R}\right)\subseteq G} is nongeneric in }G\right\} .
\]

Let
\[
\Gamma_{Ab}=\left\{ \lnot\psi\land\varphi,\theta_{n}:\psi\in\mathcal{I}_{Ab},n\in\mathbb{N}\setminus\left\{ 0\right\} \right\}
\]
 with $\theta_{n}\left(x\right)=0<x<\frac{1}{n}$ for $n\in\mathbb{N}\setminus\left\{ 0\right\} $
as above.

As before, $\Gamma_{Ab}$ is a generic partial type in $G$ and if
$q\in S^{\mathcal{R}}\left(Ab\right)$ with $q\supseteq\Gamma_{Ab}$,
then $q$ is generic in $G$.

By saturation, there is $a\models q$, thus $a$ is group-generic
of $G$ over $Ab$.

Notice that $0<b<a$, otherwise if $a<b$ , then $a\in\left(0,b\right)$
and since $\left(0,b\right)$ is an $Ab$-definable subset of $G$,
the group-genericity of $a$ implies that $\left(0,b\right)$ is generic
in $G$, but this is not possible since $b$ is infinitesimal. Therefore,
$b\in\left(0,a\right)$, which is an $Aa$-definable interval of infinitesimal
length, then $\left(0,a\right)$ cannot be generic in $G$, so $b$
is not group-generic in $G$ over $Aa$. Thus we finish Remark \ref{R:RemGenPt}.
\end{rem}

\begin{fact}\label{F:abisGen}\cite[Lemma 2.1]{Pi}
Let $G$ be a group definable over $A\subseteq M$. Let $a,b\in G$, then:
\begin{enumerate}[(i)]

\item If $a$ is generic of $G$ over $Ab$, then
$ab$ is generic of $G$ over $Ab$.

\item There are $b_{1},b_{2}\in G$ generic of
$G$ over $Ab$ such that $b=b_{1}b_{2}$.
\end{enumerate}
\end{fact}

\begin{claim}\label{C:abIsGpGen}
Let $G$ be a group definable over $A\subseteq M$. Let $a,b\in G$, then:

\begin{enumerate}[(i)]

\item If $a$ is group-generic of $G$ over $Ab$,
then $ab$ is group-generic of $G$ over $Ab$.

\item If there is a group-generic point of $G$ over
$Ab$, then there are $b_{1},b_{2}\in G$ group-generic of $G$ over
$Ab$ such that $b=b_{1}b_{2}$.

\end{enumerate}

\end{claim}

\begin{proof}

(i) We will show that $ab$ is group-generic of $G$ over $Ab$, so
let $X\subseteq G$ $Ab$-definable with $ab\in X$. Since $a\in Xb^{-1}$ and $a$ is group-generic of $G$ over $Ab$, then $Xb^{-1}$ is generic in $G$, so is $X$.

(ii) By hypothesis, there is a group-generic point $b_{1}$ of $G$
over $Ab$. Since $b_{1}^{-1}$ is also group-generic of $G$ over
$Ab$, (i) implies $b_{2}=b_{1}^{-1}b$ is group-generic of $G$ over
$Ab$, and $b=b_{1}\left(b_{1}^{-1}b\right)$.

\end{proof}

\subsection{Generic sets in the product group}\label{S:3.4}

In this subsection we prove some properties of generic definable subsets
of the product group $G\times G$ for a definably compact group $G$. Lemmas \ref{L:FubiniZgen} and \ref{L:ExistdeCaja} will be used in the proof of Theorem \ref{T:TA*}.

We first recall the notion of a Keisler measure on $G$ (which exists
by \cite[Thm. 7.7]{HPNIP2}), and a Fubini (or symmetry) Theorem (\cite[Prop. 7.5]{HPNIP2}).

\begin{dfn}\label{D:Keisler}
Let $X\subseteq M^{n}$ definable.
\begin{enumerate}[(i)]

\item A \textit{(global) Keisler measure} $\mu$ on $X$ is a finitely
additive probability measure on $Def\left(X\right)$ (the set of all
definable subsets of $X$); i.e., a map $\mu:Def\left(X\right)\rightarrow\left[0,1\right]\subseteq M$
such that $\mu\left(\emptyset\right)=0$, $\mu\left(X\right)=1$,
and for $X_{1},X_{2}\in Def\left(X\right)$, $\mu\left(X_{1}\cup X_{2}\right)=\mu\left(X_{1}\right)+\mu\left(X_{2}\right)-\mu\left(X_{1}\cap X_{2}\right)$.

\item If $\mu$ is a Keisler measure on a definable group $G$, $\mu$
is \textit{left (right) invariant} if $\mu\left(gY\right)=\mu\left(Y\right)$
($\mu\left(Yg\right)=\mu\left(Y\right)$) for every $g\in G$ and
$Y\in Def\left(G\right)$. $\mu$ is called \textit{generic} if: $\mu\left(Y\right)>0$
if and only if $Y\in Def\left(G\right)$ and $Y$ is generic in $G$.

\end{enumerate}

\end{dfn}

Note that if $\mu_{1}$ and $\mu_{2}$ are two Keisler measures on
a definable group $G$, we can define a Keisler measure $\mu_{1}\otimes\mu_{2}$
on $G\times G$ as follows: for a definable set $D\subseteq G\times G$,
$\mu_{1}\otimes\mu_{2}\left(D\right)=\int\mu_{1}\left(D_{y}\right)d\mu_{2}$
where $D_{y}=\left\{ x\in G:\left(x,y\right)\in D\right\}$. Note that function $y\mapsto\mu_{1}\left(D_{y}\right)$ is integrable with respect to $\mu_{2}$ (\cite{HPNIP2}).

The next fact gathers Proposition 7.5 and Theorem 7.7 in \cite{HPNIP2}
for the case of a definably compact group.

\begin{fact}\label{F:}\cite{HPNIP2}
Let $G$ be a definably connected definably compact group definable
in $\mathcal{M}$. Then
\begin{enumerate}[(i)]
\item $G$ has a unique left invariant (global) Keisler measure,
which is also the unique right invariant global Keisler measure on
$G$. This measure is also generic.

\item If $\mu$ and $\lambda$ are two left invariant (global) Keisler
measures on $G$, then $\mu\otimes\lambda=\lambda\otimes\mu$ and
is also left invariant.

\end{enumerate}

\end{fact}

\begin{lem}\label{L:FubiniZgen}
Let $G$ be a definably connected definably compact group definable in $\mathcal{M}$. Let $Z\subseteq G\times G$ be a definable set. For each $y\in G$, let $Z_{y}=\left\{ x\in G:\left(x,y\right)\in Z\right\}$, $Z_{gen}=\left\{ y\in G:Z_{y}\,\textrm{is generic in }G\right\}$, and $Z_{n}=\left\{ y\in G:Z_{y}\,\textrm{is}\, n\textrm{-generic in}\, G\right\} $. Then $Z$ is generic in $G\times G$ if and only if $Z_{gen}$ is generic in $G$ if and only if there is $n\in\mathbb{N}$ such that $Z_{n}$ is generic in $G$.
\end{lem}
\begin{proof}
First, observe that for every $i\in\mathbb{N}$ $Z_{i}\subseteq Z_{i+1}$, $Z_{i}$ is definable, and $Z_{gen}=\bigcup_{n\in\mathbb{N}}Z_{n}$. Let $Z_{1}^{\prime}=Z_{1}$, and $Z_{n}^{\prime}=Z_{n}\setminus Z_{n-1}$ for $n\geq2$. Then $Z_{gen}=\dot{\bigcup}_{n\in\mathbb{N}}Z_{n}^{\prime}$. By saturation, $Z_{gen}$ is generic in $G$ if only if there is $n\in\mathbb{N}$ such that $Z_{n}$ is generic in $G$ if only if there is $n\in\mathbb{N}$ such that $Z_{n}^{\prime}$ is generic in $G$.

Second, by Theorem 7.7 in \cite{HPNIP2}, $G$ has a unique generic left invariant (global) Keisler measure $\mu$. Since the Keisler measure $\mu$ is left invariant, so is the Keisler measure $\mu\otimes\mu$ on $G\times G$. Thus again by \cite[Theorem 7.7]{HPNIP2}, $\mu\otimes\mu$ is the unique left invariant Keisler measure on $G\times G$.

For a set $X$ we denote by $\bm{1}_{X}$ the indicator function of $X$; namely, $\bm{1}_{X}\left(a\right)=1$ if $a\in X$ and $\bm{1}_{X}\left(a\right)=0$ if $a\notin X$.

\begin{align*}
\mu\otimes\mu\left(Z\right) & =  \int_{G}\mu\left(Z_{y}\right)d\mu(y)\\  & =  \sum_{n\in\mathbb{N}}\int_{Z_{n}^{\prime}}\mu\left(Z_{y}\right)d\mu(y)\\  & =  \sum_{n\in\mathbb{N}}\int_{Z_{n}^{\prime}}\left(\int_{G}\bm{1}_{Z_{y}}\left(x\right)d\mu(x)\right)d\mu(y)\\  & =  \sum_{n\in\mathbb{N}}\int_{Z_{n}^{\prime}}\left(\int_{G}\bm{1}_{Z}\left(x,y\right)d\mu(x)\right)d\mu(y).
\end{align*}

By Proposition 7.5 in \cite{HPNIP2},
\begin{align*}
\int_{Z_{n}^{\prime}}\left(\int_{G}\bm{1}_{Z}\left(x,y\right)d\mu(x)\right)d\mu(y) & =  \int_{G}\left(\int_{Z_{n}^{\prime}}\bm{1}_{Z}\left(x,y\right)d\mu(y)\right)d\mu(x).
\end{align*}

Then,
\begin{align*}
 \mu\otimes\mu\left(Z\right) & = \sum_{n\in\mathbb{N}}\int_{G}\left(\int_{Z_{n}^{\prime}}\bm{1}_{Z}\left(x,y\right)d\mu(y)\right)d\mu(x)\\
 & = \sum_{n\in\mathbb{N}}\int_{G}\mu\left(Z^{x,n}\right)d\mu(x),\end{align*}  where $Z^{x,n}=\left\{y\in Z_{n}^{\prime}:\left(x,y\right)\in Z\right\}$.

Note that $Z^{x,n}\subseteq Z_{n}^{\prime}$ and $\mu\left(Z_{n}^{\prime}\right)\leq\int_{G}\mu\left(Z^{x,n}\right)d\mu(x)$, then for every $x\in G$,
\[ \mu\left(Z^{x,n}\right)\leq\mu\left(Z_{n}^{\prime}\right)\leq\int_{G}\mu\left(Z^{x,n}\right)d\mu(x).\]

From the above equations, we have that $\mu\otimes\mu\left(Z\right)>0$ $\Leftrightarrow$ $\left(\exists n\in\mathbb{N}\right)\left(\int_{G}\mu\left(Z^{x,n}\right)d\mu(x)>0\right)$ $\Leftrightarrow$ $\left(\exists n\in\mathbb{N}\right)\left(\exists x\in G\right)\left(\mu\left(Z^{x,n}\right)>0\right)$ $\Leftrightarrow$ $\left(\exists n\in\mathbb{N}\right)\left(\mu\left(Z_{n}^{\prime}\right)>0\right)$. Thus, $Z$ is generic in $G\times G$ if and only if there is $n\in\mathbb{N}$ such that $Z_{n}^{\prime}$ is generic in $G$, which is equivalent to $Z_{gen}$ is generic in $G$ by the first part of this proof.
\end{proof}

Observe that an analogous result can be proved in the same way for fibers of elements in the first component of $G\times G$.

\begin{cor}\label{C:(a,b)GpGeneric}
Let $G$ be a definably connected definably compact group definable
in $\mathcal{M}$. If $b\in G$ is group-generic of $G$ over $A$
and $a\in G$ is group-generic of $G$ over $Ab$, then $\left(a,b\right)\in G\times G$
is group-generic of $G\times G$ over $A$.

\end{cor}

\begin{proof}

Let $Z\subseteq G\times G$ $A$-definable with $\left(a,b\right)\in Z$. Since $a\in Z_{b}=\left\{ x\in G:\left(x,b\right)\in Z\right\}$,
which is $Ab$-definable, then $Z_{b}$ is generic in $G$, then $b\in Z_{gen}=\bigcup_{n\in\mathbb{N}}Z_{n}$,
where $Z_{n}=\left\{ y\in G:Z_{y}\,\textrm{is}\, n\textrm{-generic in}\, G\right\}$. Therefore, there is $n\in\mathbb{N}$ such that $b\in Z_{n}$.
As $Z_{n}$ is $A$-definable and $b$ is group-generic of $G$ over
$A$, then $Z_{n}$ is generic in $G$. By Lemma \ref{L:FubiniZgen},
this is equivalent to $Z$ is generic in $G\times G$. Then $\left(a,b\right)\in G\times G$
is group-generic of $G\times G$ over $A$.
\end{proof}

\begin{rem}\label{R:}
Let $X\subseteq M^{n}$ definable over $A\subseteq M$. If $b\in X$
is generic of $X$ over $A$ and $a\in X$ is generic of $X$ over
$Ab$, then $\left(a,b\right)\in X\times X$ is generic of $X\times X$ over $A$.
\end{rem}
\begin{proof}
It follows directly from the additivity property of the (acl-)dimension (Fact \ref{F:AlgDimProper}(ii)).
\end{proof}

Now, we will show that for a definably connected definably compact
group $G$ any definable generic set in $G\times G$ contains a definable generic box.

Recall that every Hausdorff locally compact group $G$ carries a natural
measure called the \textit{Haar measure}. A \textit{left Haar measure} $\textbf{m}$
on $G$ is a measure on the Borel algebra (namely, the $\sigma$-algebra
generated by all open sets of $G$) that is left invariant (i.e.,
$\textbf{m}\left(gX\right)=\textbf{m}\left(X\right)$ for every $g\in G$
and Borel set $X$), finite on every compact subset of $G$, and positive
for every non-empty open subset of $G$. By Haar's Theorem, $G$ has,
up to a positive multiplicative constant, a unique nontrivial left
Haar measure. If in addition $G$ is compact, then their left Haar
measures coincide with their right Haar measures, and since $\textbf{m}\left(G\right)<\infty$,
we can naturally choose a normalized Haar measure on $G$; namely,
$\textbf{m}\left(G\right)=1$.

\begin{dfn}\label{D:}
Let $G$ be a type-definable group. $G$ is \textit{compactly dominated}
by $\left(H,\textbf{m},\pi\right)$, where $H$ is a compact group,
$\textbf{m}$ is the unique normalized Haar measure on $H$ and $\pi:G\rightarrow H$
is a surjective group homomorphism, if for any definable $Y\subseteq G$
and for every $c\in H$ outside a set of $\textbf{m}$ measure zero,
either $\pi^{-1}\left(c\right)\subseteq Y$ or $\pi^{-1}\left(c\right)\subseteq G\setminus Y$;
namely,

\[
\textbf{m}\left(\left\{ c\in H:\pi^{-1}\left(c\right)\cap Y=\emptyset\textrm{ and }\pi^{-1}\left(c\right)\cap\left(G\setminus Y\right)=\emptyset\right\} \right)=0.
\]
\end{dfn}

For what follows, we recall that given a $\bigvee$-definable group $\mathcal{U}=\bigcup_{i\in I}Z_{i}$ such that $\mathcal{U}^{00}$ exists we can endow the quotient group $\mathcal{U}/\mathcal{U}^{00}$ with a topology, called the \textit{logic topology}, as follows: let
$\pi:\mathcal{U}\rightarrow\mathcal{U}/\mathcal{U}^{00}$ be the canonical
projection map and set $C\subseteq\mathcal{U}/\mathcal{U}^{00}$ to
be closed if and only if for every $i\in I\ensuremath{,}\pi^{-1}\left(C\right)\cap Z_{i}$
is type-definable. Then, by \cite[Lemma 7.5]{HPePiNIP}, these closed sets generate a locally compact
topology on $\mathcal{U}/\mathcal{U}^{00}$ making it into a Hausdorff topological
group.

\begin{fact}\label{F:}\cite{HPP2011,HPNIP2}
Let $G$ be a definably connected definably compact group definable
in $\mathcal{M}$, then $G$ is compactly dominated by $\left(G/G^{00},\textbf{m},\pi\right)$
where $m$ is the Haar measure of the compact group $G/G^{00}$ with its logic topology, and $\pi:G\rightarrow G/G^{00}$ is the canonical surjective homomorphism.
\end{fact}

\begin{lem}\label{L:ExistdeCaja}
Let $G$ be a definably connected definably compact group definable in $\mathcal{M}$. Let $Z$ be a definable generic subset in $G\times G$. Then there are definable sets $A,B\subseteq G$ generic in $G$ such that $A\times B\subseteq Z$.
\end{lem}
\begin{proof}
By \cite[Theorem 10.1]{HPP2011}, $G$ is compactly dominated by $\left(G/G^{00},\textrm{\textbf{m}},\pi\right)$ where $\textrm{\textbf{m}}$ is the Haar measure of $G/G^{00}$, and $\pi:G\rightarrow G/G^{00}$ is the canonical surjective homomorphism. And also $G\times G$ is compactly dominated by $\left(G/G^{00}\times G/G^{00},\textrm{\textbf{m}}^{\prime},\overline{\pi}\right)$ where $\textrm{\textbf{m}}^{\prime}$ is the Haar measure of $G/G^{00}\times G/G^{00}$, and $\overline{\pi}:G\times G\rightarrow G/G^{00}\times G/G^{00}$ is $\overline{\pi}=\left(\pi,\pi\right)$; i.e., $\overline{\pi}\left(x,y\right)=\left(\pi\left(x\right),\pi\left(y\right)\right)$. Note that on $\left(G\times G\right)/\left(G^{00}\times G^{00}\right)$, which can be set-theoretically identified with $G/G^{00}\times G/G^{00}$, the logic topology corresponds to the product topology on $G/G^{00}\times G/G^{00}$.

By \cite[Proposition 2.1]{Berar09}, there is $\overline{g}=\left(g_{1},g_{2}\right)\in G\times G$ such that $G^{00}\times G^{00}\subseteq Z\cdot\overline{g}$. Since $G^{00}\times G^{00}$ is a type-definable set and $Z\cdot\overline{g}$ is definable, saturation yields that there are definable $A^{*},B^{*}\subseteq G$ such that $G^{00}\subseteq A^{*},B^{*}$ and $A^{*}\times B^{*}\subseteq Z\cdot\overline{g}$. Let $A=A^{*}\cdot g_{1}^{-1}$ and $B=B^{*}\cdot g_{2}^{-1}$, then $A\times B\subseteq Z$ and $A,B$ are both definable and generic in $G$. This ends the proof of the lemma. \end{proof}

\section{A group configuration proposition for definably compact groups}\label{S:4}

\textit{From now until the end of this paper assume that $\mathcal{R}=\left(R,<,+,\cdot\right)$ is a sufficiently saturated real closed field.}

For the proof of the next proposition we will adapt the notion of
geometric structure and substructure given in \cite[Chapter 2]{HP1}.
Therefore, for the real closed field $\mathcal{R}$,
if $\mathcal{L}=\left\{ +,\cdot\right\} $, then the algebraic closure
$D=R\left(\sqrt{-1}\right)$ of $R$ is a geometric structure, and $R$ viewed as an $\mathcal{L}$-structure
is a geometric substructure of $D$ and therefore satisfies the following:

\begin{enumerate}[(i)]

\item the algebraic closures of $A\subseteq R$ in $\mathcal{R}$
in the model-theoretic and algebraic senses coincide, and for every
$A\subseteq R$ the algebraic closure of $A$ in $R$ in the sense
of the $\mathcal{L}$-structure $R$ is precisely $R\cap\textrm{acl}{}^{D}\left(A\right)$,

\item $R$ is definably closed in $D$; that is, if $b\in D$ and
$b\in\textrm{acl}{}^{D}\left(R\right)$, then $b\in R$, and

\item for each $\mathcal{L}$-formula $\varphi\left(x,y\right)$
there is some $N<\omega$ such that any model $R_{1}$ of $Th\left(R\right)$ and
every $b\in R_{1}$, if $\varphi\left(x,y\right)$ defines a finite
subset of $R_{1}$, then it defines a set with at most $N$ elements.

\end{enumerate}

We also adapt the same notation of Hrushovski and Pillay in \cite{HP1},
and we recall it below.

\begin{notation}\label{N:NotaD-R}

Let $A\subseteq R$ and $a$ a finite tuple from $R$. $\textrm{tp}\left(a/A\right)$
denotes $\textrm{tp}^{\mathcal{R}}\left(a/A\right)$, $\textrm{dcl}\left(A\right)$
denotes $\textrm{dcl}{}^{\mathcal{R}}\left(A\right)$.

$\textrm{qftp}\left(a/A\right)$ denotes $\textrm{qftp}^{D}\left(a/A\right)$
that is the set of quantifier-free $\mathcal{L}_{A}$-formulas satisfied
by $a$ in $D$. $\textrm{qfdcl}\left(A\right)$ denotes $\textrm{qfdcl}^{D}\left(A\right)$
that is the set of elements of $D$ definable over $A$ by quantifier-free
formulas, but since $R$ is definably closed in $D$, so $\textrm{qfdcl}^{D}\left(A\right)\subseteq R$.
Note that since $D$ has quantifier elimination, $\textrm{qfdcl}^{D}\left(A\right)=\textrm{dcl}^{D}\left(A\right)$.
Finally, by $\textrm{acl}\left(A\right)$ we denote $\textrm{acl}^{D}\left(A\right)$.

\end{notation}

Recall that a group $H$ definable in an algebraically closed field $D$ is $D$-\textit{definably connected} if there is no proper nontrivial $D$-definable subgroup of $H$ of finite index.

\begin{prop}\label{P:3.1*}
Let $D$ the algebraic closure of $\mathcal{R}$. Let $G$ be a definably compact definably connected group definable in $\mathcal{R}$. Then there are a finite subset $A\subseteq R$ over which $G$ is defined, a $D$-definably connected group $H$ definable in $D$ over $A$, points $a,b,c$ of $G$ and points $a^{\prime},b^{\prime},c^{\prime}$ of $H\left(R\right)$ such that
\begin{enumerate}[(i)]
\item $a\cdot b=c$ (in G) and $a^{\prime}\cdot b^{\prime}=c^{\prime}$ (in H),
\item $\textrm{acl}\left(aA\right)=\textrm{acl}\left(a^{\prime}A\right)$, $\textrm{acl}\left(bA\right)=\textrm{acl}\left(b^{\prime}A\right)$ and $\textrm{acl}\left(cA\right)=\textrm{acl}\left(c^{\prime}A\right)$,
\item $b$ is a group-generic point of $G$ over $A$, $a$ is a group-generic point of $G$ over $Ab$,
\item $a^{\prime}$ and $b^{\prime}$ are generic points of $H\left(R\right)$ over $A$ and are independent with each other over $A$.
\end{enumerate}
Note that $a^{\prime}$ and $b^{\prime}$ are only generic and not group-generic.
\end{prop}
\begin{proof}
This proof is essentially the same as that of \cite[Proposition 3.1]{HP1}, what
is new is that we have to prove that the points $a$ and $b$ introduced
below remain group-generic of $G$ over each of the sets of parameters
defined by Hrushovski and Pillay in their proof. To achieve this we
summarize without proof the unmodified parts of the proof of \cite[Proposition 3.1]{HP1}
and just focuses on the new parts. We refer the reader to \cite{HP1}
for appropriate model-theoretic background.

The first part of the proof of Proposition 3.1 in \cite{HP1} is devoted
to yield a set-up in which \cite[Proposition 1.8.1]{HP1} can be applied, and
get with this the existence of the connected group $H$ definable
in $D$ mentioned in the conclusion of Proposition 3.1. This is done
through a series of lemmas and observations.

Let us start with a finite subset $A_0$ of $R$ over which $G$ and its group operation are defined. Let $\dim\left(G\right)=n$. By Proposition \ref{P:SupL1}(i), there is $b\in G$ group-generic of $G$ over $A_0$. By Prop. \ref{P:SupL1}(ii) and saturation, there is $a\in G$ group-generic of $G$ over $A_{0}\cup\left\{b\right\}$, then $a\forkindep[A_0]b$. Let $c=a\cdot b$, then, by Claim \ref{C:abIsGpGen}, $c$ is group-generic of $G$ over $A_{0}\cup\left\{b\right\}$. And also $\dim\left(a,b,c/A_{0}\right)=\dim\left(a,b/A_{0}\right)=2n$.

In $\mathcal{R}$, $c\in\textrm{dcl}\left(a,b,A_{0}\right)$ and $b\in\textrm{dcl}\left(a,c,A_{0}\right)$. Thus we start with three group-generic points of $G$ such that each two of them are independent (over some set of parameters) and define the third in $\mathcal{R}$. As Hrushovski and Pillay point out in \cite{HP1}, the key is to modify those points by points in $R$ such that two of them define the third in the structure $D$; namely, that dcl is replaced by qfdcl in order to lay the foundations to apply \cite[Prop. 1.8.1]{HP1}.

\begin{lem}\label{L:3.2*}
There are a finite subset $A_{2}$ of $R$, containing $A_{0}$, and tuples $a_{1},b_{1,},c_{1}$ in $R$ such that
\begin{enumerate}[(i)]
\item $b$ and $a$ are group-generic of $G$ over $A_{2}$ and $A_{2}b$, respectively,
\item $\textrm{acl}\left(a,A_{2}\right)=\textrm{acl}\left(a_{1},A_{2}\right)$, $\textrm{acl}\left(b,A_{2}\right)=\textrm{acl}\left(b_{1},A_{2}\right)$, $\textrm{acl}\left(c,A_{2}\right)=\textrm{acl}\left(c_{1},A_{2}\right)$,
\item $b_{1}\in\textrm{qfdcl}\left(a_{1},c_{1},A_{2}\right)$, and $c_{1}\in\textrm{qfdcl}\left(a_{1},b_{1},A_{2}\right)$.
\end{enumerate}
\end{lem}

\begin{proof}
The only thing we need to prove here is the existence of elements $x^{\prime}$, $z_1$ from $R$ satisfying the same conditions of the $x^{\prime}$ and $z_1$ of Hrushovski and Pillay in their original proof of \cite[Lemma 3.2]{HP1} such that  $b$ and $a$ are group-generic of $G$ over $A_{2}=A_{0}x^{\prime}z_1$ and $A_{2}b$, respectively, because the rest of the proof is exactly the same as that of \cite[Lemma 3.2]{HP1}.

\begin{claim}\label{C:}

\begin{enumerate}[(i)]

\item There is a generic point $x^{\prime}$ of $G$ over $A_{0}ab$
such that if $A_{1}=A_{0}x^{\prime}$, then $b$ is group-generic
of $G$ over $A_{1}$ and $a$ is group-generic of $G$ over $A_{1}b$.

\item Consider $A_{1}$ as in (i), then there is a generic point
$z_{1}$ of $G$ over $A_{1}ab$ such that if $A_{2}=A_{1}z_{1}$,
then $b$ is group-generic of $G$ over $A_{2}$ and $a$ is group-generic
of $G$ over $A_{2}b$.

\end{enumerate}

\end{claim}

\begin{proof}

(i) Since b is group-generic of $G$ over $A_{0}$ and $a$ is group-generic
of $G$ over $A_{0}b$, then Corollary \ref{C:SupL2}(vi) yields the
existence of a generic element $x^{\prime}$ of $G$ over $A_{0}ab$
such that $b$ is group-generic of $G$ over $A_{0}x^{\prime}$ and
$a$ is group-generic of $G$ over $A_{0}x^{\prime}b$.

(ii) From the conclusion of (i) and Corollary \ref{C:SupL2}(vi), we
get the existence of a generic point $z_{1}$ of $G$ over $A_{1}ab$
such that $b$ is group-generic of $G$ over $A_{1}z_{1}$ and $a$
is group-generic of $G$ over $A_{1}z_{1}b$.

\end{proof}
This ends the proof of Lemma \ref{L:3.2*}.
\end{proof}

Now, let $a_{1},b_{1},c_{1}$ and $A_{2}$ be as given by Lemma \ref{L:3.2*}, and $A=\textrm{acl}\left(A_{2}\right)\cap R$. Therefore, $a_{1},b_{1},c_{1}$ each have dimension $n$ over $A$. Since $\textrm{acl}\left(A_{2}\right)\cap R=\textrm{acl}{}^{\mathcal{R}}\left(A_{2}\right)$, then Remark \ref{R:ggACL} implies that $b$ and $a$ are also group-generic of $G$ over $A$ and $Ab$, respectively.

\cite[Remark 3.3]{HP1} yields that $\textrm{qftp}\left(b_{1},c_{1}/A,a_{1}\right)$
is stationary, and hence we can define the canonical base $\sigma$
of $\textrm{qftp}\left(b_{1},c_{1}/A,a_{1}\right)$. Then $\sigma\in\textrm{qfdcl}\left(A,a_{1}\right)$.
Since $R$ is definably closed in $D$, $\sigma\in R$.

Let \[
r=\textrm{qftp}\left(\sigma/A\right),\; q_{1}=\textrm{qftp}\left(b_{1}/A\right),\;\textrm{and}\; q_{2}=\textrm{qftp}\left(c_{1}/A\right),
\] then $\dim\left(q_{1}\right)=\dim\left(q_{2}\right)=n$.

By Remark 3.4 in \cite{HP1}, $r$ is stationary, $\dim\left(r\right)=n$,
$\sigma\forkindep[A]b_{1}$, $\sigma\forkindep[A]c_{1}$, $b_{1}\in\textrm{qfdcl}\left(\sigma,c_{1},A\right)$,
and $c_{1}\in\textrm{qfdcl}\left(\sigma,b_{1},A\right)$. Therefore,
there is some $A$-definable partial function in the sense of $D$,
say $\mu$, such that $c_{1}=\mu\left(\sigma,b_{1}\right)$. And note
that whenever $\sigma^{\prime}\models r$ and $b_{1}^{\prime}\models q_{1}$
with $\sigma^{\prime}\forkindep[A]b_{1}^{\prime}$, then $\mu\left(\sigma^{\prime},b_{1}^{\prime}\right)$
is well-defined, realises $q_{2}$ and is independent with each of
$\sigma^{\prime},b_{1}^{\prime}$ over $A$. Similarly, as $b_{1}\in\textrm{qfdcl}\left(\sigma,c_{1},A\right)$,
there is some $A$-definable partial function in the sense of $D$,
say $\upsilon$, such that $b_{1}=\upsilon\left(\sigma,c_{1}\right)$.
And note that whenever $\sigma^{\prime}\models r$ and $c_{1}^{\prime}\models q_{2}$
with $\sigma^{\prime}\forkindep[A]c_{1}^{\prime}$, then $\upsilon\left(\sigma^{\prime},c_{1}^{\prime}\right)$
is well-defined, realises $q_{1}$ and is independent with each of
$\sigma^{\prime},c_{1}^{\prime}$ over $A$.

Now, let $\sigma_{1},\sigma_{2}\models r$ with $\sigma_{1}\forkindep[A]\sigma_{2}$ and $\sigma_{1},\sigma_{2}\in R$. Let $b_{2}\models q_{1}$ such that $b_{2}\forkindep[A]\left\{ \sigma_{1},\sigma_{2}\right\}$ and $b_{2}\in R$.

Then $\mu\left(\sigma_{1},b_{2}\right)$ is defined, realises $q_{2}$,
and is independent with $\sigma_{2}$ over $A$. Therefore, $\upsilon\left(\sigma_{2},\mu\left(\sigma_{1},b_{2}\right)\right)$
is defined and realises $q_{1}$. Denote $\upsilon\left(\sigma_{2},\mu\left(\sigma_{1},b_{2}\right)\right)$
by $b_{3}$.

By Remark 3.6 in \cite{HP1}, $b_{3}\in\textrm{qfdcl}\left(\sigma_{1},\sigma_{2},b_{2},A\right)$,
$b_{2}\in\textrm{qfdcl}\left(\sigma_{1},\sigma_{2},b_{3},A\right)$,
each of $b_{2},b_{3}$ is independent with $\left\{ \sigma_{1},\sigma_{2}\right\} $
over $A$, and $\textrm{qftp}\left(b_{2},b_{3}/\sigma_{1},\sigma_{2},A\right)$
is stationary. Then we can define the canonical base of $\textrm{qftp}\left(b_{2},b_{3}/\sigma_{1},\sigma_{2},A\right)$
and denote it by $\tau$. Then $\tau\in\textrm{qfdcl}\left(\sigma_{1},\sigma_{2},A\right)$, so $\tau\in R$.

Let $s=\textrm{qftp}\left(\tau/A\right)$. By \cite[Lemma 3.7]{HP1}, $\dim\left(s\right)=n$. As was proved in Remark 3.4 in \cite{HP1}, we have that $b_{3}\in\textrm{qfdcl}\left(\tau,b_{2},A\right)$,
$b_{2}\in\textrm{qfdcl}\left(\tau,b_{3},A\right)$; moreover, $\tau$
is independent with each of $b_{2},b_{3}$ over $A$. Therefore, there
is some $A$-definable partial function $\mu^{\prime}$ in the sense
of $D$ such that $b_{3}=\mu^{\prime}\left(\tau,b_{2}\right)$, and
whenever $\tau^{\prime}\models s$ and $b_{1}^{\prime}\models q_{1}$
with $\tau^{\prime}\forkindep[A]b_{1}^{\prime}$, then $\mu^{\prime}\left(\tau^{\prime},b_{1}^{\prime}\right)$
is well-defined and realises $q_{1}$.

At this stage Hrushovski and Pillay obtain two $n$-dimensional stationary types $s$
and $q_{1}$ over $A$ that satisfy the hypothesis of \cite[Proposition 1.8.1]{HP1}.
Moreover, the functions $f$ and $g$, which are quantifier-free definable in $D$ over $A$, in the hypothesis of Prop. 1.8.1 in \cite{HP1} correspond to the functions $f$ in \cite[Lemma 3.8]{HP1} and $\mu^{\prime}$, respectively.

Next comes the application of \cite[Prop. 1.8.1]{HP1}. Let $H$, $X$, $h_1$, and $h_2$ as given by Prop. 1.8.1 in \cite{HP1}. We can assume that $h_1,h_2$ are both the identity function. Thus $H$ is a connected group definable in $D$ over $A$ with generic type $s$, $X$ is a set definable in $D$ over $A$ with generic type $q_1$, and there is a transitive group action $\Lambda:H\times X\rightarrow X:\left(h,x\right)\mapsto\Lambda\left(h,x\right)$, which is also definable in $D$ over $A$.

Note that since $\tau\in H\left(R\right)$, $\tau\models s$, and $\mathcal{R}$ viewed as an $\left\{ +,\cdot\right\}$- structure is a geometric substructure of $D$, then $\dim\left(H\left(R\right)\right)=n$.
Similarly, from $b_{1}\in X\left(R\right)$ and $b_{1}\models q_{1}$,
we have $\dim\left(X\left(R\right)\right)=n$.

Moreover, we have the following:

\begin{enumerate}[(i)]

\item for $\tau_{1},\tau_{2}\models s$ with $\tau_{1}\forkindep[A]\tau_{2}$,
the product $\tau_{1}\cdot\tau_{2}$ in the group $H$ is exactly
$f\left(\tau_{1},\tau_{2}\right)$, and

\item for any $\tau\models s$ and $b\models q_{1}$ with $\tau\forkindep[A]b$,
$\Lambda\left(\tau,b\right)$ is exactly $\mu^{\prime}\left(\tau,b\right)$.

\end{enumerate}

Since $\tau\in\textrm{qfdcl}\left(\sigma_{1},\sigma_{2},A\right)$,
there is some $A$-definable partial function $\xi$ in the sense
of $D$ such that $\tau=\xi\left(\sigma_{1},\sigma_{2}\right)$. And
note that whenever $\sigma_{1}^{\prime},\sigma_{2}^{\prime}\models r$
with $\sigma_{1}^{\prime}\forkindep[A]\sigma_{2}^{\prime}$, then
$\xi\left(\sigma_{1}^{\prime},\sigma_{2}^{\prime}\right)$ is well-defined,
realises $s$ and is independent with each of $\sigma_{1}^{\prime},\sigma_{2}^{\prime}$
over $A$.

Finally, in the last part of this proof we introduce some new sets of parameters, define the points $a^{\prime},b^{\prime},c^{\prime}$ generic in $H\left(R\right)$, and prove some interalgebraicity between
them and the points $a,b,c$.

Let $\sigma, b_1, c_1$ be as fixed after the proof of Lemma \ref{L:3.2*}, which are all in $R$.

\begin{claim}\label{C:ClaimProbl}
There is a tuple $\sigma_{1}$ from $R$ such that $\textrm{qftp}\left(\sigma_{1}/A\right)=\textrm{qftp}\left(\sigma/A\right)$,
$\ensuremath{\sigma_{1}\forkindep[A]\left\{ \sigma,b_{1},c_{1}\right\} }$,
and $b$ and $a$ are group-generic of $G$ over $\textrm{acl}\left(A\sigma_{1}\right)\cap R$
and $\left(\textrm{acl}\left(A\sigma_{1}\right)\cap R\right)\cup\left\{ b\right\} $,
respectively.
\end{claim}

\begin{proof}
Since $r=\textrm{qftp}\left(\sigma/A\right)$ is stationary and
every definable set in $\mathcal{R}$ has generic points in $R$,
then there is $\widetilde{\sigma_{1}}$ a tuple from $R$ such that
$\textrm{qftp}\left(\widetilde{\sigma_{1}}/A\right)=\textrm{qftp}\left(\sigma/A\right)$
and $\widetilde{\sigma_{1}}\forkindep[A]\left\{ \sigma,b_{1},c_{1}\right\} $.
Then, by Corollary \ref{C:SupL2}(v), there is $\sigma_{1}\subseteq R$ such that $\textrm{tp}\left(\widetilde{\sigma_{1}}/A\right)=\textrm{tp}\left(\sigma_{1}/A\right)$, $b$
and $a$ are group-generic of $G$ over $A\sigma_{1}$ and $A\sigma_{1}b$,
respectively. We will see that $\sigma_{1}$ satisfies the same properties
of $\widetilde{\sigma_{1}}$.

First, since $\textrm{tp}\left(\widetilde{\sigma_{1}}/A\right)=\textrm{tp}\left(\sigma_{1}/A\right)$ and $\textrm{qftp}\left(\widetilde{\sigma_{1}}/A\right)=\textrm{qftp}\left(\sigma/A\right)$, so $\textrm{qftp}\left(\sigma_{1}/A\right)=\textrm{qftp}\left(\sigma/A\right)$. Second, from the construction throughout the proof of \cite[Proposition 3.1]{HP1}, we have that $\left\{ \sigma,b_{1},c_{1}\right\} $ and $\left\{ a,b\right\} $ are interalgebraic over $A$. Then, $\sigma_{1}\forkindep[A]\left\{ \sigma,b_{1},c_{1}\right\} $ if
and only if $\ensuremath{\sigma_{1}\forkindep[A]\left\{ a,b\right\} }$. And note that since $b\forkindep[A]\sigma_{1}$ and $a\forkindep[Ab]\sigma_{1}$, then $\sigma_{1}\forkindep[A]\left\{ a,b\right\}$.

By Remark \ref{R:ggACL}, if $b$ and $a$ are group-generic of
$G$ over $A\sigma_{1}$ and $A\sigma_{1}b$, respectively, then $b$
and $a$ are group-generic of $G$ over $\textrm{acl}\left(A,\sigma_{1}\right)\cap R$
and $\left(\textrm{acl}\left(A,\sigma_{1}\right)\cap R\right)\cup\left\{ b\right\} $,
respectively. This completes the proof of the claim.
\end{proof}

Let $\sigma_{1}$ be as given by Claim \ref{C:ClaimProbl}. Then $\upsilon\left(\sigma_{1},c_{1}\right)\models q_{1}$, is in
$R$, and $\upsilon\left(\sigma_{1},c_{1}\right)\forkindep[A]\sigma_{1}$.
Let $c_{2}=\upsilon\left(\sigma_{1},c_{1}\right)$. Also, we get $\xi\left(\sigma_{1},\sigma\right)\models s$, $\xi\left(\sigma_{1},\sigma\right)\forkindep[A]\sigma_{1}$,
and is in $H\left(R\right)$. Let $\tau=\xi\left(\sigma_{1},\sigma\right)$,
and $A_{1}=\textrm{acl}\left(A,\sigma_{1}\right)\cap R$. Then so far we have that:

\begin{enumerate}[(i)]

\item $b$ and $a$ are group-generic of $G$ over $A_{1}$ and $A_{1}b$,
respectively,
\item $\textrm{acl}\left(A_{1},a\right)=\textrm{acl}\left(A_{1},\tau\right)$, $\textrm{acl}\left(A_{1},c\right)=\textrm{acl}\left(A_{1},c_{2}\right)$, and
\item $\tau\models s$, $b_{1}\models q_{1}$, and $c_{2}\models q_{1}$.

\end{enumerate}

We complete the proof of this proposition below.


Since $R$ is definably closed in $D$, for every $\tau^{\prime}\in H\left(R\right)$ and every $\beta\in X\left(R\right)$, $\Lambda\left(\tau^{\prime},\beta\right)\in X\left(R\right)$. Moreover, $H\left(R\right)$ acts on $X\left(R\right)$ by the group action $\Lambda$ restricted to $H\left(R\right)\times X\left(R\right)$,
which is definable in $R$ over $A$.

Let us define a relation $\sim$ on $X\left(R\right)$. For $\beta_{1},\beta_{2}\in X\left(R\right)$
we say $\beta_{1}\sim\beta_{2}$ if and only if $\beta_{1}$ and $\beta_{2}$ are both in the same $H\left(R\right)$-orbit, namely if $\beta_{1}\in \Lambda\left(H\left(R\right),\beta_2\right)$. Then $\sim$
is an equivalence relation on $X\left(R\right)$ definable in $R$ over $A\subseteq A_{1}$.

Since $R$ is o-minimal and has elimination of imaginaries, \cite[Corollary 4.7]{P} implies that there are at most finitely many $\sim$-classes whose dimension equals to $\dim\left(X\left(R\right)\right)$. Therefore, for every $\beta$ generic of $X\left(R\right)$ over $A_{1}$, the equivalence class of $\beta$ under $\sim$, denoted $\left[\beta\right]$, has dimension $n$ and is a definable set in $R$ over $A\subseteq A_{1}$.


Now, recall that $b_{1}$ is generic of $X\left(R\right)$ over $A_1$ and $b_{1}\models q_{1}$,
then $\left[b_{1}\right]$ is an $n$-dimensional set definable in $\mathcal{R}$ over $A_{1}$.

Let $b_{2}^{\prime}$ be a generic element of $\left[b_{1}\right]$
over $A_{1}$. Then by Corollary \ref{C:SupL2}(v), there is a tuple $b_{2}$ from $R$ such that $\textrm{tp}\left(b_{2}/A_{1}\right)=\textrm{tp}\left(b_{2}^{\prime}/A_{1}\right)$,
and $b$ and $a$ are group-generic of $G$ over $A_{1}b_{2}$ and
$A_{1}bb_{2}$, respectively. Since $b_{2}^{\prime}\in\left[b_{1}\right]$
and $\left[b_{1}\right]$ is defined over $A_{1}$, then $b_{2}\in\left[b_{1}\right]$.
Thus, there is $\tau_{1}\in H\left(R\right)$ such that $b_{2}=\Lambda\left(\tau_{1}^{-1},b_{1}\right)$.

Since $a$ is also generic of $G$ over $A_{1}bb_{2}$, then $a\forkindep[A_{1}]\left\{ b,b_{2}\right\}$;
thus $b_{2}\forkindep[A_{1},b]a$. Also, as $b$ is also generic of
$G$ over $A_{1}b_{2}$, $b_{2}\forkindep[A_{1}]b$. Therefore, $b_{2}\forkindep[A_{1}]\left\{ a,b\right\} $,
thus $b_{2}$ is generic of $X\left(R\right)$ over $A_{1}ab$.

Now, since $\textrm{acl}\left(A_{1},a\right)=\textrm{acl}\left(A_{1},\tau\right)$
and $\textrm{acl}\left(A_{1},b\right)=\textrm{acl}\left(A_{1},b_{1}\right)$,
then $b_{2}\forkindep[A_{1}]\left\{ \tau,b_{1}\right\} $.

\begin{claim}\label{C:}

\begin{enumerate}[(i)]

\item $\tau_{1}$ is generic of $H\left(R\right)$ over $A_{1}\tau b_{1}$,
and

\item $\tau_{1}$ is generic of $H\left(R\right)$ over $A_{1}\tau b_{2}$.

\item $\tau$ is generic of $H\left(R\right)$ over $A_{1}\tau_{1} b_{2}$.

\end{enumerate}

\end{claim}

\begin{proof}

(i) Note that

\begin{align*}
\dim\left(\tau_{1},b_{2}/A_{1},\tau,b_{1}\right) & =  \dim\left(b_{2}/A_{1},\tau,b_{1}\right)+\dim\left(\tau_{1}/A_{1},\tau,b_{1},b_{2}\right)\\
 & =  n+\dim\left(\tau_{1}/A_{1},\tau,b_{1},b_{2}\right).
\end{align*}

Also, $\dim\left(\tau_{1},b_{2}/A_{1},\tau,b_{1}\right)=\dim\left(\tau_{1}/A_{1},\tau,b_{1}\right)$
since $b_{2}=\Lambda\left(\tau_{1}^{-1},b_{1}\right)$. Therefore,
$\dim\left(\tau_{1}/A_{1},\tau,b_{1}\right)=n$.

(ii) First, observe that as $b_{2}\in\textrm{acl}\left(A_{1},\tau_{1},b_{1}\right)$,
then
\begin{align*}
\dim\left(\tau_{1},b_{2},b_{1}/A_{1},\tau\right) & =  \dim\left(\tau_{1},b_{1}/A_{1},\tau\right)\\
 & =  \dim\left(\tau_{1}/A_{1},\tau,b_{1}\right)+\dim\left(b_{1}/A_{1},\tau\right)\\
 & =  2n.
\end{align*}

Also,
\begin{align*}
\dim\left(\tau_{1},b_{2},b_{1}/A_{1},\tau\right) & =  \dim\left(b_{2}/A_{1},\tau\right)+\dim\left(\tau_{1}/A_{1},\tau,b_{2}\right)\\
 & =  n+\dim\left(\tau_{1}/A_{1},\tau,b_{2}\right).
\end{align*}

Hence, $\dim\left(\tau_{1}/A_{1},\tau,b_{2}\right)=n$.

(iii) By (ii), $\tau_{1}\forkindep[A_{1},b_{2}]\tau$, and since $b_{2}\forkindep[A_{1}]\tau$,
then $\tau\forkindep[A_{1}]\left\{ b_{2},\tau_{1}\right\}$. This finishes this proof.

\end{proof}

\begin{claim}\label{C:}

\begin{enumerate}[(i)]

\item $\textrm{acl}\left(A_{1},b_{2},a\right)=\textrm{acl}\left(A_{1},b_{2},\tau\right)$,

\item $\textrm{acl}\left(A_{1},b_{2},b\right)=\textrm{acl}\left(A_{1},b_{2},\tau_{1}\right)$,
and

\item $\textrm{acl}\left(A_{1},b_{2},c\right)=\textrm{acl}\left(A_{1},b_{2},\tau\cdot\tau_{1}\right)$.

\end{enumerate}

\end{claim}

\begin{proof}

(i) It follows from $\textrm{acl}\left(A_{1},a\right)=\textrm{acl}\left(A_{1},\tau\right)$.

(ii) First, we will see that $\tau_{1}\in\textrm{acl}\left(A_{1},b_{2},b_{1}\right)$.
\begin{align*}
\dim\left(\tau_{1},b_{2}/A_{1},b_{1}\right) & =  \dim\left(b_{2}/A_{1},b_{1}\right)+\dim\left(\tau_{1}/A_{1},b_{2},b_{1}\right)\\
 & =  n+\dim\left(\tau_{1}/A_{1},b_{2},b_{1}\right).
\end{align*}

Also,
\begin{align*}
\dim\left(\tau_{1},b_{2}/A_{1},b_{1}\right) & =  \dim\left(\tau_{1}/A_{1},b_{1}\right)+\dim\left(b_{2}/A_{1},b_{1},\tau_{1}\right)\\
 & =  n.
\end{align*}
Then, $\tau_{1}\in\textrm{acl}\left(A_{1},b_{2},b_{1}\right)$, and
since $b$ and $b_{1}$ are interalgebraic over $A_{1}$, so $\textrm{acl}\left(A_{1},b_{2},\tau_{1}\right)\subseteq\textrm{acl}\left(A_{1},b_{2},b\right)$.

Additionally, $b_{1}=\Lambda\left(\tau_{1},b_{2}\right)$, then $b_{1}\in\textrm{acl}\left(A_{1},b_{2},\tau_{1}\right)$. So $\textrm{acl}\left(A_{1},b_{2},b\right)\subseteq\textrm{acl}\left(A_{1},b_{2},\tau_{1}\right)$.
This completes (ii).

(iii) First, we will see that $\Lambda\left(\tau\cdot\tau_{1},b_{2}\right)=c_{2}$.
From the properties of the action and the maps $\nu,\mu$, and $\xi$,
we have
\begin{align*}
\Lambda\left(\tau\cdot\tau_{1},b_{2}\right) & =  \Lambda\left(\tau\cdot\tau_{1},\Lambda\left(\tau_{1}^{-1},b_{1}\right)\right)\\
 & =  \Lambda\left(\tau,\Lambda\left(\tau_{1}\cdot\tau_{1}^{-1},b_{1}\right)\right)\\
 & =  \Lambda\left(\tau,b_{1}\right),
\end{align*}
\begin{align*}
c_{2} & =  \nu\left(\sigma_{1},c_{1}\right)\\
 & =  \nu\left(\sigma_{1},\mu\left(\sigma,b_{1}\right)\right)\\
 & =  \Lambda\left(\xi\left(\sigma_{1},\sigma\right),b_{1}\right)\\
 & =  \Lambda\left(\tau,b_{1}\right).
\end{align*}
Then, $\Lambda\left(\tau\cdot\tau_{1},b_{2}\right)=c_{2}$.

We will see that $\tau\cdot\tau_{1}\in\textrm{acl}\left(A_{1},b_{2},c_{2}\right)$.
By Fact \ref{F:abisGen}, since $\tau$ is generic of $H\left(R\right)$
over $A_{1} \tau_{1} b_{2}$, then $\tau\cdot\tau_{1}$ is generic
of $H\left(R\right)$ over $A_{1} \tau_{1} b_{2}$.

Now, since $b_{2}\forkindep[A_{1}]\left\{ a,b\right\} $ and $\textrm{acl}\left(A_{1},c\right)=\textrm{acl}\left(A_{1},c_{2}\right)$,
$c_{2}\forkindep[A_{1}]b_{2}$, and thus we get

\begin{align*}
\dim\left(\tau\cdot\tau_{1},c_{2}/A_{1},b_{2}\right) & =  \dim\left(c_{2}/A_{1},b_{2}\right)+\dim\left(\tau\cdot\tau_{1}/A_{1},b_{2},c_{2}\right)\\
 & =  n+\dim\left(\tau\cdot\tau_{1}/A_{1},b_{2},c_{2}\right).
\end{align*}

Also,
\begin{align*}
\dim\left(\tau\cdot\tau_{1},c_{2}/A_{1},b_{2}\right) & =  \dim\left(\tau\cdot\tau_{1}/A_{1},b_{2}\right)+\dim\left(c_{2}/A_{1},b_{2},\tau\cdot\tau_{1}\right)\\
 & =  n.
\end{align*}

Thus, $\tau\cdot\tau_{1}\in\textrm{acl}\left(A_{1},b_{2},c_{2}\right)$,
and thus $\textrm{acl}\left(A_{1},b_{2},\tau\cdot\tau_{1}\right)\subseteq\textrm{acl}\left(A_{1},b_{2},c\right)$.
Finally, as $c_{2}=\Lambda\left(\tau\cdot\tau_{1},b_{2}\right)$ and
$\textrm{acl}\left(A_{1},c\right)=\textrm{acl}\left(A_{1},c_{2}\right)$,
then $\textrm{acl}\left(A_{1},b_{2},c\right)\subseteq\textrm{acl}\left(A_{1},b_{2},\tau\cdot\tau_{1}\right)$,
this ends this proof.

\end{proof}

Let $A_{2}=\textrm{acl}\left(A_{1},b_{2}\right)\cap R$, $a^{\prime}=\tau$, $b^{\prime}=\tau_{1}$, and $c^{\prime}=a^{\prime}\cdot b^{\prime}$ the product  of $a^{\prime}$ by $b^{\prime}$ in $H$. So far we have proved that:

\begin{enumerate}[(i)]

\item By Remark \ref{R:ggACL}, $b$ and $a$ are group-generic of
$G$ over $A_{2}$ and $A_{2}b$, respectively,

\item $a^{\prime}$ and $b^{\prime}$ are generic of $H\left(R\right)$
over $A_{2}$ and $a^{\prime}\forkindep[A_{2}]b^{\prime}$,

\item $\textrm{acl}\left(A_{2},a\right)=\textrm{acl}\left(A_{2},a^{\prime}\right)$,
$\textrm{acl}\left(A_{2},b\right)=\textrm{acl}\left(A_{2},b^{\prime}\right)$,
and $\textrm{acl}\left(A_{2},c\right)=\textrm{acl}\left(A_{2},c^{\prime}\right)$.

\end{enumerate}

Finally, let $A$ any finite subset of $A_{2}$ over which $G$ and $H$ are defined with the obtained properties. This concludes the proof of Proposition \ref{P:3.1*}.

\end{proof}

\section{A local homomorphism with generic domain between a semialgebraically compact semialgebraic group over R and the R-points of an R-algebraic group}\label{S:5}

In this section we prove the main theorem of this paper:

\begin{thm}\label{T:TA*}
Let $G$ be a definably compact definably connected group definable in $\mathcal{R}$. Then there are
\begin{enumerate}[(i)]
    \item a connected $R$-algebraic group $H$ such that $\dim\left(G\right)=\dim\left(H\left(R\right)\right)=\dim\left(H\right)$,
    \item a definable $X\subseteq G$ such that $G^{00}\subseteq X$,
    \item a definable homeomorphism $\phi:X\subseteq G\rightarrow \phi\left(X\right)\subseteq H\left(R\right)$ such that $\phi$ and $\phi^{-1}$ are local homomorphisms.
\end{enumerate}
\end{thm}
\begin{proof}
Denote by $D$ the algebraic closure of $R$. By Proposition \ref{P:3.1*}, there are a finite subset $A\subseteq R$ over which
$G$ is defined, a $D$-definably connected group $H$ quantifier-free $A$-definable
in $D$, a group-generic point $b$ of $G$ over $A$, a group-generic point $a$ of $G$ over $Ab$, thus $c=a\cdot b$ is also group-generic of $G$ over $Ab$ (this by Claim \ref{C:abIsGpGen}), as well as points $a^{\prime},b^{\prime},c^{\prime}= a^{\prime}\cdot b^{\prime}\in H\left(R\right)$ generic in $H\left(R\right)$ over $A$ with the properties given there. Let $k$ be the subfield generated by $A$.

As every $D$-definably connected group definable over $k$ in the algebraic
closed field $D$ is definably isomorphic over $k$ to a connected
$k$-algebraic group (\cite{BousII,LouChunkThm}), we may assume that
$H$ is such algebraic group. Moreover, by the conditions of the points $a,b,c$, $a^{\prime},b^{\prime},c^{\prime}$
of Proposition \ref{P:3.1*}, the dimension of $H$ as algebraic
group is equal to the o-minimal dimensions $\dim\left(G\right)$ and $\dim\left(H\left(R\right)\right)$.

Since $a$ and $a^{\prime}$ are interalgebraic over $k$ in $R$
and $R$ is o-minimal, $a$ and $a^{\prime}$ are interdefinable over
$k$ in $R$, and similarly for $b,b^{\prime}$ and $c,c^{\prime}$. From now on, we work in $R$ and by definable we will mean $R$-definable.

By \cite[Lemma 4.8(i)]{HP1} (which holds for $R$ instead of $\mathbb{R}$), there are open $k$-definable neighbourhoods $U,V$ and $W$ in $G$
of $a,b,c$, respectively, and $U^{\prime}$, $V^{\prime}$, $W^{\prime}$
in $H\left(R\right)$ of $a^{\prime},b^{\prime},c^{\prime}$, respectively,  and
$k$-definable functions $f,g$, and $h$ such that
$f\left(a\right)=a^{\prime}$ and $f$ is a definable homeomorphism
between $U$ and $U^{\prime}$, $g\left(b\right)=b^{\prime}$ and
$g$ is a definable homeomorphism between $V$ and $V^{\prime}$, and $h\left(c\right)=c^{\prime}$
and $h$ is a definable homeomorphism between $W$ and $W^{\prime}$.

Let \[Z=\left\{ \left(x,y\right)\in G\times G:x\in U,y\in V,x\cdot y\in W,f\left(x\right)\cdot g\left(y\right)=h\left(x\cdot y\right)\right\}.\]

Since $b$ is group-generic in $G$ over $k$ and $a$ is group-generic
in $G$ over $kb$, Corollary \ref{C:(a,b)GpGeneric} yields $\left(a,b\right)$
is group-generic in $G\times G$ over $k$. Thus, as $Z$ is $k$-definable
and $\left(a,b\right)\in Z$, then $Z$ is generic in $G\times G$.

By Lemma \ref{L:ExistdeCaja}, there are definable sets $A,B$ generic
in $G$ such that $A\times B\subseteq Z$.

\begin{claim}

Let $X,Y$ definable sets generic in $G$. Then there is $g\in G$ such that $X\cap\left(Y\cdot g^{-1}\right)$ is generic in $G$ and that $\left(X\cap\left(Y\cdot g^{-1}\right)\right)\cdot g\subseteq Y$.

\end{claim}
\begin{proof}
By genericity of $X$ in $G$, there are $g_{1},\ldots,g_{k}\in G$
such that $G=\bigcup_{i\leq k}X\cdot g_{i}$, hence $Y=\bigcup_{i\leq k}\left(X\cdot g_{i}\right)\cap Y=\bigcup_{i\leq k}\left(X\cap\left(Y\cdot g_{i}^{-1}\right)\right)\cdot g_{i}$.
Since $Y$ is generic, there is $i\leq k$ such that $\left(X\cap\left(Y\cdot g_{i}^{-1}\right)\right)\cdot g_{i}$ is generic in $G$. Thus, with $g=g_{i}$ we get the desired result.
\end{proof}

Then by the above claim applied to $A^{-1}$ and $B$, there is $g\in G$ such that $\left(A^{-1}\cap\left(B\cdot g^{-1}\right)\right)\cdot g$ is generic and contained in $B$. Then if $A^{\prime}=A\cap\left(g\cdot B^{-1}\right)$, so $A^{\prime}$ is generic in $G$ and $\left(A^{\prime}\right)^{-1}\cdot g\subseteq B$. By \cite[Proposition 2.1]{Berar09}, there is $s\in A^{\prime}$ such that $s\cdot G^{00}\subseteq A^{\prime}$.
Since $A^{\prime}=A\cap\left(g\cdot B^{-1}\right)$, $s=g\cdot b^{-1}$
for some $b\in B$. Let $t=b$. Note that $s\cdot t\in A^{\prime}\cdot B\subseteq W$. So far we have shown that there are generic
sets $A^{\prime}$ and $B$ in $G$ such that

\begin{enumerate}

\item[(i)] $A^{\prime}\times B\subseteq Z$ , and

\item[(ii)] There is $\left(s,t\right)\in A^{\prime}\times B$ such
that $s\cdot G^{00}\subseteq A^{\prime}$ and $\left(A^{\prime}\right)^{-1}\cdot\left(s\cdot t\right)\subseteq B$.

\end{enumerate}

Let $X=\left(A^{\prime}\right)^{-1}\cdot s$, then $G^{00}\subseteq X$. Finally, we will define the local homomorphism.

\begin{prop}\label{P:4.9}

The definable homeomorphism
\[
\phi:X=\left(A^{\prime}\right)^{-1}\cdot s\rightarrow\left(f(A^{\prime})\right)^{-1}\cdot f\left(s\right)
\]
defined by $\phi\left(x^{-1}\cdot s\right)=f\left(x\right)^{-1}\cdot f\left(s\right)$
for $x\in A^{\prime}$, and its inverse
\[
\phi^{-1}:\left(f(A^{\prime})\right)^{-1}\cdot f\left(s\right)\rightarrow X,
\]
 which is given by $\phi^{-1}\left(y^{-1}\cdot f\left(s\right)\right)=\left(f^{-1}\left(y\right)\right)^{-1}\cdot s$
for $y\in f(A^{\prime})$, are local homomorphisms between $G$ and
$H\left(R\right)^{0}$.

\end{prop}
\begin{proof}
First, note the following.

\begin{claim}\label{C:}

\begin{enumerate}[(i)]

\item If $\left(x_{1}^{-1}\cdot s\right)\cdot\left(x_{2}^{-1}\cdot s\right)\in\left(A^{\prime}\right)^{-1}\cdot s$,
then

\begin{align*}
\phi\left(\left(x_{1}^{-1}\cdot s\right)\cdot\left(x_{2}^{-1}\cdot s\right)\right) & =  \phi\left(x_{1}^{-1}\cdot s\right)\cdot\phi\left(x_{2}^{-1}\cdot s\right)\Leftrightarrow\\
f\left(\left(x_{1}^{-1}\cdot s\cdot x_{2}^{-1}\right)^{-1}\right)^{-1}\cdot f\left(s\right) & =  f\left(x_{1}\right)^{-1}\cdot f\left(s\right)\cdot f\left(x_{2}\right)^{-1}\cdot f\left(s\right)\Leftrightarrow\\
f\left(x_{1}\right)f\left(\left(x_{1}^{-1}\cdot s\cdot x_{2}^{-1}\right)^{-1}\right)^{-1}f\left(x_{2}\right) & =  f\left(s\right).
\end{align*}

\item If $\left(f\left(x_{1}\right)^{-1}\cdot f\left(s\right)\right)\cdot\left(f\left(x_{2}\right)^{-1}\cdot f\left(s\right)\right)\in f\left(Y\right)^{-1}\cdot f\left(s\right)$,
then

$\phi^{-1}\left(\left(f\left(x_{1}\right)^{-1}\cdot f\left(s\right)\right)\cdot\left(f\left(x_{2}\right)^{-1}\cdot f\left(s\right)\right)\right)=\phi^{-1}\left(f\left(x_{1}\right)^{-1}\cdot f\left(s\right)\right)\cdot\phi^{-1}\left(f\left(x_{2}\right)^{-1}\cdot f\left(s\right)\right)\Leftrightarrow$

\begin{align*}
\left(f^{-1}\left(\left(f\left(x_{1}\right)^{-1}\cdot f\left(s\right)\cdot f\left(x_{2}\right)^{-1}\right)^{-1}\right)\right)^{-1}\cdot s & =  x_{1}^{-1}\cdot s\cdot x_{2}^{-1}\cdot s\Leftrightarrow\\
x_{1}\cdot\left(f^{-1}\left(\left(f\left(x_{1}\right)^{-1}\cdot f\left(s\right)\cdot f\left(x_{2}\right)^{-1}\right)^{-1}\right)\right)^{-1}\cdot x_{2} & =  s.
\end{align*}

\end{enumerate}

\end{claim}

\begin{proof}

It follows directly from the definitions of $\phi$ and $\phi^{-1}$.

\end{proof}

We will show in Claim \ref{C:theproof} that each of $\phi$ and $\phi^{-1}$
satisfy one of the equivalent conditions formulated above. To prove
Claim \ref{C:theproof} we will use the next technical fact.

From now on, let $s^{\prime}=f\left(s\right)$ and $t^{\prime}=g\left(t\right)$.

\begin{claim}\label{C:Propchiinv}
\begin{enumerate}
\item[(i)] For every $y_{1}\in f\left(A^{\prime}\right)$ and every $y_{2}\in g\left(B\right)$, $f^{-1}\left(y_{1}\right)\cdot g^{-1}\left(y_{2}\right)=h^{-1}\left(y_{1}\cdot y_{2}\right)$.
\item[(ii)] $\left(f\left(A^{\prime}\right)\right)^{-1}\cdot\left(s^{\prime}\cdot t^{\prime}\right)\subseteq g\left(B\right)$. \end{enumerate} \end{claim}
\begin{proof}
(i) Let $y_{1}\in f\left(A^{\prime}\right)$, so there is $x_{1}\in A^{\prime}$ such that $f\left(x_{1}\right)=y_{1}$. Let $y_{2}\in g\left(B\right)$, so there is $x_{2}\in B$ such that $g\left(x_{2}\right)=y_{2}$. Since $A^{\prime}\times B\subseteq Z$, then $y_{1}\cdot y_{2}=f\left(x_{1}\right)\cdot g\left(x_{2}\right)=h\left(x_{1}\cdot x_{2}\right)$; therefore, $h^{-1}\left(y_{1}\cdot y_{2}\right)=x_{1}\cdot x_{2}=f^{-1}\left(y_{1}\right)\cdot g^{-1}\left(y_{2}\right)$.

(ii) Let $x\in A^{\prime}$, then $x^{-1}\cdot\left(s\cdot t\right)\in B$, so $g\left(x^{-1}\cdot\left(s\cdot t\right)\right)=\left(f\left(x\right)\right)^{-1}\cdot h\left(s\cdot t\right)=\left(f\left(x\right)\right)^{-1}\cdot\left(s^{\prime}\cdot t^{\prime}\right)$. Hence, $\left(f\left(A^{\prime}\right)\right)^{-1}\cdot\left(s^{\prime}\cdot t^{\prime}\right)\subseteq g\left(B\right)$.
\end{proof}

\begin{claim}\label{C:theproof}

Let $x_{1},x_{2}\in A^{\prime}$.

\begin{enumerate}[(i)]

\item Let $z^{-1}=x_{1}^{-1}\cdot s\cdot x_{2}^{-1}$. If $z^{-1}\in\left(A^{\prime}\right)^{-1}$,
then $f\left(x_{1}\right)f\left(z\right)^{-1}f\left(x_{2}\right)=f\left(s\right)=s^{\prime}$.

\item Let $w^{-1}=f\left(x_{1}\right)^{-1}\cdot s^{\prime}\cdot f\left(x_{2}\right)^{-1}$.
If $w^{-1}\in f\left(A^{\prime}\right)^{-1}$, then $x_{1}\cdot\left(f^{-1}\left(w\right)\right)^{-1}\cdot x_{2}=s=f^{-1}\left(s^{\prime}\right)$.

\end{enumerate}

\end{claim}

\begin{proof}

For the following recall that $\left(s,t\right)\in A^{\prime}\times B$,
and $A^{\prime}\times B\subseteq Z$, so for every $\left(x,y\right)\in A^{\prime}\times B$,
$f\left(x\right)\cdot g\left(y\right)=h\left(x\cdot y\right)$.

(i) \begin{align*}
f\left(s\right)\cdot g\left(t\right) & =  h\left(s\cdot t\right)\\
 & =  h\left(\left(x_{1}\cdot z^{-1}\cdot x_{2}\right)\cdot t\right)\\
 & =  h\left(x_{1}\cdot\left(z^{-1}\cdot x_{2}\cdot t\right)\right),\:\textrm{since}\:\ensuremath{z^{-1}\cdot x_{2}\cdot t=x_{1}^{-1}\cdot s\cdot t\in\left(A^{\prime}\right)^{-1}\cdot s\cdot t\subseteq B},\\
 & =  f\left(x_{1}\right)\cdot g\left(z^{-1}\cdot x_{2}\cdot t\right)\\
 & =  f\left(x_{1}\right)\cdot f\left(z\right)^{-1}\cdot h\left(x_{2}\cdot t\right)\\
 & =  f\left(x_{1}\right)\cdot f\left(z\right)^{-1}\cdot f\left(x_{2}\right)\cdot g\left(t\right).
\end{align*}

After cancelling $g\left(t\right)$, the desired conclusion is obtained.

(ii) By Claim \ref{C:Propchiinv}, we have the next equations.
\begin{align*}
f^{-1}\left(s^{\prime}\right)\cdot g^{-1}\left(t^{\prime}\right) & =  h^{-1}\left(s^{\prime}\cdot t^{\prime}\right)\\
 & =  h^{-1}\left(\left(f\left(x_{1}\right)\cdot w^{-1}\cdot f\left(x_{2}\right)\right)\cdot t^{\prime}\right)\\
 & =  h^{-1}\left(f\left(x_{1}\right)\cdot\left(w^{-1}\cdot f\left(x_{2}\right)\cdot t^{\prime}\right)\right),\:\textrm{since}\:\ensuremath{w^{-1}\cdot f\left(x_{2}\right)\cdot t^{\prime}=f\left(x_{1}\right){}^{-1}\cdot s^{\prime}\cdot t^{\prime}\in g\left(B\right)},\\
 & =  f^{-1}\left(f\left(x_{1}\right)\right)\cdot g^{-1}\left(w^{-1}\cdot f\left(x_{2}\right)\cdot t^{\prime}\right)\\
 & =  x_{1}\cdot\left(f^{-1}\left(w\right)\right)^{-1}\cdot h^{-1}\left(f\left(x_{2}\right)\cdot t^{\prime}\right)\\
 & =  x_{1}\cdot\left(f^{-1}\left(w\right)\right)^{-1}\cdot x_{2}\cdot g^{-1}\left(t^{\prime}\right).
\end{align*}

After cancelling $g^{-1}\left(t^{\prime}\right)$, we conclude the claim.
\end{proof}
This finishes the proof of Proposition \ref{P:4.9}.
\end{proof}
Theorem \ref{T:TA*} is proved.
\end{proof}

\begin{claim}\label{C:G00Berar}
Let $G$ be a definably connected definably compact group definable in a sufficiently saturated o-minimal expansion of a real closed field. Let $X\subseteq G$ definable with $G^{00}\subseteq X$. Then there are definable sets $X_1,X_2 \subseteq G$ such that $X_1$ is definably simply connected, $X_2$ is definably connected and symmetric, and $G^{00}\subseteq X_1\subseteq X_2\subseteq X$.
\end{claim}
\begin{proof}
By saturation, $G^{00}=\bigcap_{i\in\mathbb{N}}X_{i}=\bigcap_{i\in\mathbb{N}}X_{i}\cdot X_{i}^{-1}$,
and since $G^{00}\subseteq X$, then there is $i\in\mathbb{N}$ such
that $G^{00}\subseteq X_{i}\cdot X_{i}^{-1}\subseteq X$. By the Cell decomposition Theorem (\cite{LVD}), $X_{i}$
is a finite union of definably simply connected cells, thus one of them has
to be generic, call it $C$. By \cite[Proposition 2.1]{Berar09},
there is $g\in G$ such that $G^{00}\subseteq C\cdot g\subseteq C\cdot C^{-1}\subseteq X_{i}\cdot X_{i}^{-1}\subseteq X$.
Finally, let $X_1=C\cdot g$ and $X_2=C\cdot C^{-1}$.
\end{proof}

\begin{rem}\label{R:TA*simm-connected}
By Claim \ref{C:G00Berar}, the definable generic set $X$ of Theorem \ref{T:TA*} can be taken either definably connected, symmetric, and $G^{00}\subseteq X$, or definably simply connected and $G^{00}\subseteq X$.
\end{rem}

Let $X$ be the set in the conclusion of Theorem \ref{T:TA*}. As
$G^{00}\subseteq X$, then there is $X^{\prime}\subseteq X$ symmetric
such that $G^{00}\subseteq X^{\prime}\subseteq X^{\prime}\cdot X^{\prime}\subseteq X$.
Then the next result holds in $\mathcal{R}$:

\begin{cor}\label{C:ThmA*for Transfer}

Let $G$ be a definably compact definably connected group definable in $\mathcal{R}$. Then there are
\begin{enumerate}[(i)]
\item a connected $R$-algebraic group $H$ such that $\dim\left(G\right)=\dim\left(H\left(R\right)\right)=\dim\left(H\right)$,
\item definable sets $X^{\prime},X\subseteq G$ such that $X^{\prime}$ is a symmetric neighborhood of the identity of $G$ and generic in $G$, and $X^{\prime}\cdot X^{\prime}\subseteq X$,
\item a definable homeomorphism $\phi:X\subseteq G\rightarrow \phi\left(X\right)\subseteq H\left(R\right)$ such that $\phi$ and $\phi^{-1}$ are local homomorphisms.
\end{enumerate}

\end{cor}

By transferring Corollary \ref{C:ThmA*for Transfer} from $\mathcal{R}$ to any real closed field, we have that Corollary \ref{C:ThmA*for Transfer} holds in any real closed field, not necessarily sufficiently saturated.

\section*{Acknowledgements}
I would like to express my gratitude to the Universidad de los Andes,
Colombia and the University of Haifa, Israel for supporting and funding my research as well as for their stimulating hospitality. I would also like to thank warmly to my advisors: Alf Onshuus and Kobi
Peterzil for their support, generous ideas, and kindness during this work.

I want to express my gratitude to Anand Pillay for suggesting to Kobi Peterzil the problem discussed in this work: the study of semialgebraic groups over a real closed field. Also, thanks to the Israel-US Binational Science Foundation for their support.

The main results of this paper have been presented on the winter of 2016 at the Logic Seminar of the Institut Camille Jordan, Université Claude Bernard - Lyon 1 (Lyon) and at the Oberseminar Modelltheorie of the Universit{\"a}t Konstanz (Konstanz).

\nocite{ErraEdCov,BaEdErrata,PetStein99,HP2,TentZie,Hodges}
\bibliographystyle{plain}
\bibliography{IP}

\end{document}